\documentclass[a4paper,12pt]{amsart}

\usepackage{amsfonts}
\usepackage{amssymb}
\usepackage{graphics}

\newcommand{\dl}{\lambda}

\newcommand{\C}{\mathbb{C}}
\newcommand{\N}{\mathbb{N}}
\newcommand{\Z}{\mathbb{Z}}
\newcommand{\R}{\mathbb{R}}

\newcommand{\fol}{\mathcal{F}}

\newcommand{\calO}{{\mathcal{O}}}

\newcommand{\Diff}{{{\rm Diff}\, ({\mathbb C^n}, 0)}}

\newcommand{\diff}{{{\rm Diff}\, ({\mathbb C}, 0)}}
\newcommand{\diffn}{{{\rm Diff}\, ({\mathbb C^n}, 0)}}

\newcommand{\ghatXone}{{\widehat{\mathfrak{X}}}}
\newcommand{\ghatX}{{\widehat{\mathfrak{X}}_2}}
\newcommand{\formdiffn}{{\widehat{\rm Diff}_{1} (\mathbb{C}^2, 0)}}
\newcommand{\formalC}{{\mathbb{C} [[x,y]]}}
\newcommand{\fieldC}{{\mathbb{C} ((x,y))}}
\newcommand{\diffCtwo}{{{\rm Diff}_1 ({\mathbb C}^2, 0)}}

\newcommand{\Diffgentwo}{{{\rm Diff}\, ({\mathbb C}^2, 0)}}

\def\picill#1by#2(#3)#4
{\vbox to #2
{\hrule width #1 height 0pt depth 0pt
\vfill\special{illustration #3 scaled #4}}}

\newtheorem{theorem}{Theorem}
\newtheorem{prop}{Proposition}

\newtheorem{lemma}{Lemma}

\newtheorem{obs}{Remark}

\theoremstyle{definition}
\newtheorem{defi}{Definition}

\theoremstyle{remark}

\begin{document}

\title[Topological dynamics in $\Diff$]{Discrete orbits and special subgroups of $\Diff$}

\author{Julio C. Rebelo \, \, \, \& \, \, \, Helena Reis}
\address{}
\thanks{}

\begin{abstract}
The local topological dynamics of subgroups of $\diffn$, with special emphasis on $\Diffgentwo$, is discussed with a view towards integrability questions. It is proved
in particular that a subgroup of $\Diffgentwo$ possessing locally finite orbits is necessarily solvable. Other results and examples related
to higher-dimensional generalizations of Mattei-Moussu's celebrated topological characterization of integrability are also provided.
These examples also settle a fundamental question raised by the previous work of Camara-Scardua.
\end{abstract}

\maketitle

\section{Introduction}

In many senses, this paper is motivated by the celebrated topological characterization of integrable holomorphic
vector fields/foliation in dimension~$2$ obtained in \cite{M-M}. The fundamental issue singled out
in \cite{M-M} being the fact that the existence of holomorphic first integrals can be read off the topological dynamics
of the foliation, and in particular of its holonomy pseudogroup. It then follows that the existence of (non-constant) holomorphic
first integrals is a property invariant by topological conjugation.

Only recently, however, extensions of Mattei-Moussu's results
to higher dimensions have started being investigated, partly due to recent progress made in the understanding of the local
dynamics of diffeomorphisms of $(\C^2,0)$ tangent to the identity, cf. \cite{abate}, \cite{BM}, \cite{raissy} and their references.
One basic question was to know whether a (local) singular holomorphic foliation
on $(\C^n,0)$ topologically conjugate to another holomorphic foliation possessing $n-1$ independent holomorphic first integrals should
possess $n-1$ independent holomorphic first integrals as well. Though an affirmative answer seemed to be expected, in 
\cite{thesis}, the authors exhibited two topologically
conjugate foliations on $(\C^3, 0)$ such that one admits two independent holomorphic first integrals but not
the other. It became then clear that the extension of Mattei-Moussu theorem to higher dimensions was a far more subtle problem.

Ultimately the purpose of this paper is to contribute to the understanding of the above mentioned problem by presenting ``counterexamples'' as well as
affirmative statements that parallel some fundamental results known in the classical low-dimensional case. These results concern
either the topological dynamics of subgroups of ${\rm Diff}\, (\C^2,0)$ (sometimes ${\rm Diff}\, (\C^n,0)$) or, in the case of Theorem~B below,
singular foliations with a simple singularity as previously discussed in \cite{scardua}. Indeed, the main result of \cite{scardua} as well as
the question left open about ``holonomy maps with non-isolated fixed point'' were the starting points of the present work.

To state our main results, let us place ourselves in the context of pseudogroups of $\diffn$ or of $\Diffgentwo$
(the reader may check Section~2 for further details). Our first result is a simple elaboration
of the corresponding statement in \cite{M-M} that turns out to generalize the corresponding result in \cite{scardua} as it dispenses with the use of the deep
theorem on parabolic domains, valid in dimension~$2$ and due to M. Abate \cite{abate}.

\vspace{0.2cm}

\noindent {\bf Theorem~A}. {\sl Let $G \subset \diffn$ be a finitely generated pseudogroup on a small neighborhood of the origin in $\C^n$.
Given $g \in G$, let ${\rm Dom}\, (g)$ denote the domain of definition of $g$ as element of the pseudogroup in question.
Suppose that for every $g \in G$ and $p \in {\rm Dom}\, (g)$ satisfying $g(p) =p$, one of the following holds:
either $p$ is an isolated fixed point of $g$ or $g$ coincides with the identity on a neighborhood of $p$. Then the pseudogroup $G$
has finite orbits on a neighborhood of the origin if and only if $G$ itself is finite.}

\vspace{0.2cm}

\noindent {\bf Remark}. {\rm When $G$ is a subgroup of ${\rm Diff}\, (\C,0)$ the assumption of Theorem~A is automatically
verified so that the statement is reduced to Mattei-Moussu's corresponding result in \cite{M-M}. On the other hand, it is proved in
\cite{M-M} that a subgroup of ${\rm Diff}\, (\C,0)$ is not only finite but also cyclic. In full generality the second part of the statement
cannot be generalized to higher dimensions since every finite group embeds into a matrix group of sufficiently high dimension. In Section~2
the reader will find simple examples showing that, in fact, the group need not be cyclic already in dimension~$2$ and even if the additional
assumption of Theorem~A is satisfied.}

\vspace{0.2cm}

In \cite{scardua} this statement is essentially proved in dimension~$2$ by resorting to Abate's theorem, cf. Section~2 for a detailed comparison
between the corresponding statements. The authors of \cite{scardua} then go ahead to turn their statement
into an application about ``complete integrability'' of differential equations. A similar application holds in arbitrary dimensions as it will
be seen in Theorem~B below. Before stating Theorem~B, it is however convenient to mention a minor issue
concerning the formulation of the results in \cite{scardua}, as pointed
out by Y. Genzmer in  his review to the article in question. In fact, the authors have failed to mention they assumed the
corresponding holonomy maps to have only isolated fixed points, whereas the assumption is clearly needed from the corresponding proof.
In particular, if no assumption concerning isolated fixed points is put forward, it becomes unclear whether or not
a cyclic (pseudo-) group having finite orbits should be finite and whether a corresponding Siegel singularity associated to a holonomy map with finite orbits
must be ``completely integrable''. Though an affirmative answer to the latter question was expected, as pointed out by Genzmer \cite{yohann}, both statements
turned out to be false as it follows from the ``Complement to the Theorem~B'' below. The reader will find below accurate statements in these directions.

\vspace{0.2cm}

\noindent {\bf Theorem~B}. {\sl Let $\fol$ be a singular foliation associated to a holomorphic vector field $X$ with an isolated singularity at the
origin. Suppose that the linear part of $X$ has determinant different from zero, belongs to the Siegel domain and satisfy the conditions~3 and~4 in Section~2.
Suppose also that the holonomy map associated to each separatrix of $\fol$ has finite orbits and 
satisfies the conditions of Theorem~A (in the sense of the cyclic (pseudo-) group they generate). Then $\fol$ admits $n-1$ independent holomorphic first integrals.}

\vspace{0.2cm}

\noindent {\bf Remark}. {\rm Concerning the statement of Theorem~B, note that in the case~$n=3$, the above mentioned conditions~3 and~4 in
Section~2 are automatically verified provided that $X$ is of {\it strict Siegel type}. In particular Theorem~B covers the main result in \cite{scardua}.
In particular it suffices to have finite orbit for the holonomy map associated to {\it one of the separatrices of $\fol$}. 
Also the reader will find in Section~2 a slightly more accurate version of the statement of Theorem~B showing, in particular, that the assumption
concerning holonomy maps has to be verified for a certain separatrix of $\fol$, i.e. it does not have to be checked over all the separatrices}.

As mentioned, if the main assumption in Theorem~A is dropped, then the conclusions will not longer hold. In particular, we shall prove the following:

\vspace{0.2cm}

\noindent {\bf Complement to Theorem~B}. {\sl Let $\fol$ denote the foliation associated to the vector field
\[
X = x(1 + x^2yz^3) \frac{\partial }{\partial x} + y(1 - x^2yz^3) \frac{\partial }{\partial y} - z \frac{\partial }{\partial z} \, .
\]
The foliation $\fol$ does not possess two independent holomorphic first integrals (though it possesses one non-constant holomorphic first
integral). Besides the holonomy map associated to the axis $\{ x=y=0\}$ has finite orbits whereas it does not generate a finite subgroup of
$\Diffgentwo$.}

\vspace{0.2cm}

In view of the previous examples, and given that the assumption on isolated fixed points of Theorem~A is not fully satisfactory,
we may go back to the fundamental question about finding the ``correct'' generalization of Mattei-Moussu theorem. In other words, what type of algebraic
conditions a subgroup of $\diffn$ possessing finite orbits should verify. To simplify the discussion, we shall content ourselves of dealing with
subgroups of $\Diffgentwo$. Differential Galois theory, as well as Morales-Ramis theory cf. \cite{moralesetal}, suggests that a (pseudo-) group having
finite orbits may be solvable. The main result of this paper, namely Theorem~C below, confirms this suggestion. In fact, it suffices to deal with
pseudo-groups possessing {\it locally discrete orbits}\, (also called locally finite orbits) which allows us to apply the statement also to foliations
admitting meromorphic first integrals. Naturally a point $p$ is said to have locally discrete orbit, or equivalently, locally finite
orbit, under a (pseudo-) group $G$ if, for every
point $q$ in the $G$-orbit $G.p$ of $p$, there exists a neighborhood $U$ of $q$ such that the set $U \cap G.p$ is finite. A group is said to have
locally finite orbits if all its orbits are locally finite. With this terminology, we state:

\vspace{0.2cm}

\noindent {\bf Theorem~C}. {\sl Suppose that $G$ is a finitely generated (pseudo-) subgroup of $\Diffgentwo$ with locally discrete orbits. Then $G$ is
solvable.}

\vspace{0.2cm}

Theorem~C is the most elaborate result of this paper and to our knowledge the first general result concerning the dynamics of solvable/non-solvable
subgroups of $\diffn$ for $n \geq 2$. Several comments are needed to properly place this statement in perspective concerning previous works. First the
highly developed case $n=1$ must be singled out. In this case, the structure of solvable subgroups of ${\rm Diff}\, (\C,0)$ is well-understood
\cite{cerveaumoussu}, \cite{russians}
and, formally, our statement is a consequence of the much stronger results of Shcherbakov and Nakai \cite{shcherbakov}, \cite{nakai}. For higher dimensions,
however, there are new dynamical phenomena related, for instance, to the existence of non-solvable discrete subgroups of ${\rm GL}\, (n , \C)$, $n \geq 2$.
These phenomena prevent us from extending the results of Shcherbakov and Nakai without additional assumptions. In fact, the ``sharp'' extension of their theory
to higher dimensions remains and outstanding problem despite some significant progress made in \cite{belliart}, \cite{lorayandI}.

Another point to be made about Theorem~C is that its proof does not rely on any typically two-dimensional phenomenon and hence can probably be
extended to arbitrary dimensions, though we have not pursued this direction. Theorem~C is actually constituted by two main ideas
which nicely complement each other. On one hand, there is the standard theory of Kleinian groups and stable manifolds
that essentially allows us to reduce the proof of
the theorem in question to the case of subgroups of $\diffCtwo$, the group of holomorphic
diffeomorphisms tangent to the identity. To deal with the latter group, we then adapt the ``recurrence theorem''
established by Ghys in \cite{ghysBSBM} by means of his notion of ``pseudo-solvable'' group. An important remark concerning this adaptation is that
elements of $\Diffgentwo$ tangent to the identity are automatically ``close to the identity'' in a sense suitable to ensure convergence of
sequences of commutators.

Still considering the use Ghys's ideas to subgroups of $\diffCtwo$, it is necessary to clarify the connection between solvable
and pseudo-solvable subgroups of $\diffCtwo$. For $n=1$, Ghys showed in \cite{ghysBSBM} that these notions coincide and this result is extended
to $n=2$ here. Although the strategy followed is similar to the employed by Ghys, this extension is not immediate since the structure of solvable
subgroups of $\Diffgentwo$ is not nearly as developed as in the one-dimensional case. We are therefore led to work out several algebraic aspects
on the solvable subgroups of $\Diffgentwo, \, \diffCtwo$ and to deal with the new phenomena concerning existence of non-constant first integrals
and/or with rank~$2$ abelian groups.

Theorem~C raises a number of interesting questions, in particular connections with Morales-Sim\'o-Ramis theory \cite{moralesetal} seems very
promising. Another more specific question that may turn out to be quite deep concerns the classification of solvable non-abelian subgroups of
$\Diffgentwo$ possessing locally finite orbits. The reader is reminded that, for $n=1$, the corresponding result is due to Birkhoff, though it was independently
re-discovered by Loray in \cite{Loray}. Since this beautiful result possesses a number of applications, we believe
that its generalization to dimension~$2$ is a problem worth further investigation.

Let us finish this Introduction with an outline of the structure of this paper. Section~2 contains the proofs of Theorems~A and~B along with the
relevant definitions. As mentioned, Theorem~A is a simple elaboration of the arguments in \cite{M-M}. In turn, Theorem~B is an application
of Theorem~A going through useful results due to P. Elizarov-Il'yashenko and to Reis, \cite{EI}, \cite{helena}. Section~3 contains a few
interesting examples of local dynamics of diffeomorphisms tangent to the identity along with local foliations realizing some of them as local holonomy
map. In particular the example appearing in the Complement to Theorem~B is detailed so as to settle the main issue left open in \cite{scardua}.
Section~4 is the most technical part of the paper. It involves a detailed algebraic study of abelian and solvable (formal) groups of germs 
of diffeomorphisms in dimension~$2$. Campbell-Hausdorff type formulas are widely used in this study which, ultimately, aims at showing
that ``pseudo-solvable'' subgroups of $\diffCtwo$ are, actually, solvable (Proposition~\ref{commuting9}). A reader willing to take for grant
Proposition~\ref{commuting9} or, alternatively, content himself/herself with a statement involving ``pseudo-solvable'' groups in Theorem~C can
skip the whole of Section~4. Yet it may be pointed out that the algebraic description of solvable subgroups of $\diffCtwo$ developed
in the course of the mentioned section is original and likely to have further interest. Finally in Section~5, ideas from Ghys \cite{ghysBSBM} are combined to
Proposition~\ref{commuting9} and to standard results on Kleinian groups to yield the proof of Theorem~C.

\section{Theorems~A and~B}

In the sequel, $G$ denotes a finitely generated subgroup of $\diffn$, where $\diffn$ stands for the group of germs of
local holomorphic diffeomorphisms of $\C^n$ fixing the origin. Assume then that $G$
is generated by the elements $h_1, \ldots, h_k \in \diffn$. A natural way to make sense of the local dynamics of $G$ consists
of choosing representatives for $h_1, \ldots, h_k$ as local diffeomorphisms fixing $0 \in \C$. These representatives are still denoted
by $h_1, \ldots, h_k$ and, once this choice is made, $G$ itself can be identified to the {\it pseudogroup}\,
generated by these local diffeomorphisms on a (sufficiently small) neighborhood of the origin. It is then convenient to begin by
briefly recalling the notion of {\it pseudogroup}. For this, consider a small neighborhood $V$ of the origin where the
local diffeomorphisms $h_1, \ldots, h_k$, along with their inverses $h_1^{-1}, \ldots, h_k^{-1}$, are defined and one-to-one.
The pseudogroup generated by $h_1, \ldots, h_k$ (or rather by $h_1, \ldots , h_k, h_1^{-1}, \ldots, h_k^{-1}$ if there is any risk of confusion)
on $V$ is defined as follows. Every element of $G$ has the form $F = F_s \circ \ldots \circ F_1$ where each $F_i$,
$i \in \{1, \ldots, s\}$, belongs to the set $\{h_i^{\pm 1}, i=1, \ldots, k\}$. The element $F \in G$ should be regarded as a one-to-one holomorphic map
defined on a subset of $V$. Indeed, the domain of definition of $F = F_s \circ \ldots \circ F_1$, as an
element of the pseudogroup, consists of those points $x \in V$ such that for every $1 \leq l < s$ the point $F_l \circ
\ldots \circ F_1(x)$ belongs to $V$. Since the origin is fixed by the diffeomorphisms $h_1, \ldots, h_k$, it follows that the
domain of definition of every element $F$ is a non-empty open set containing
the origin. This open set may however be disconnected. Whenever no misunderstanding is possible, the pseudogroup defined above will also
be denoted by $G$ and we are allowed to shift back and forward from $G$ viewed as pseudogroup or as group of germs.

Let us continue with some definitions that will be useful throughout the text. Suppose we are given local holomorphic diffeomorphisms $h_1,
\ldots, h_k, h_1^{-1}, \ldots, h_k^{-1}$ fixing the origin of $\C^n$. Let $V$ be a neighborhood of the origin where all these local diffeomorphisms
are defined and one-to-one. From now on, let $G$ be viewed as the pseudogroup acting on $V$ generated by these local diffeomorphisms.
Given an element $h \in G$, the domain of definition of $h$ (as element of $G$) will be denoted by ${\rm Dom}_V (h)$.

\begin{defi}
The $V_G$-orbit $\calO_V^G (p)$ of a point $p \in V$ is the set of points in $V$ obtained from $p$ by taking its image through every element of $G$ whose
domain of definition (as element of $G$) contains $p$.
In other words,
\[
\calO_V^G (p) = \{q \in V \; \, ; \; \,  q = h(p), \; h \in G \; \; {\rm and} \; \; p \in {\rm Dom}_V (h) \} \, .
\]
Fixed $h \in G$, the $V_h$-orbit of $p$ can be defined as the $V_{<h>}$-orbit of $p$, where $<h>$ denotes the subgroup
of $\diffn$ generated by $h$.
\end{defi}

We can now define ``pseudogroups with finite orbits" and ``pseudogroups with locally discrete orbits''.

\begin{defi}\label{def_finiteorbits}
A pseudogroup $G \subseteq \diffn$ is said to have finite orbits if there exists a sufficiently small open neighborhood $V$ of $0
\in \C^n$, where $h_1, \ldots , h_k, h_1^{-1}, \ldots, h_k^{-1}$ are well-defined injective maps, such that the set $\calO_V^G (p)$
is finite for every $p\in V$. Analogously, $h \in G$ is said to have finite orbits if the pseudogroup $\langle h \rangle$
generated by $h$ has finite orbits.

Similarly, a pseudogroup is said to have locally discrete orbits (or equivalently locally finite orbits)
if for every $p \in V$ and for every point $q \in \calO_V^G (p)$,
there exists a neighborhood $W \subset \C^n$ of $q$ such that $W \cap \calO_V^G (p) = \{ q \}$.
\end{defi}

Fixed $h \in G$, the {\it number of iterations of $p$ by $h$}\, is the cardinality of the set $\{ n \in \Z \; \, ; \; \, p \in {\rm Dom}_V (h^{n}) \}$,
where ${\rm Dom}_V (h^{n})$ stands for the domain of definition of $h^n$ as element of the pseudogroup in question.
The number of iterations of $p$ by $h$ is denoted by $\mu_V^h (p)$ and belongs to $\N \cup \{\infty\}$. The lemma below is attributed
to Lewowicz and it can be found in \cite{M-M}.

\begin{lemma}[Lewowicz]\label{lemmalewowicz}
Let $K$ be a compact connected neighborhood of $0 \in \R^n$ and $h$ a homeomorphism from $K$ onto $h(K) \subseteq \R^n$
verifying $h(0) = 0$. Then there exists a point $p$ on the boundary $\partial K$ of $K$ whose number of iterations
in $K$ by $h$ is infinite, i.e. $p$ satisfies $\mu_K^h (p) = \infty$.\qed
\end{lemma}

Fixed an open set $V$,  note that the existence of points in $V$ such that $\mu_K^h (p) = \infty$ does not imply that
$p$ is a point with infinite orbit, i.e. there may exist points $p$ in $V$ such that $\mu_V^h(p)=\infty$ but $\#
\calO_V^{<h>}(p)<\infty$, where $\#$ stands for the cardinality of the set in question.
These points are called {\it periodic for $h$ on $V$}. A local diffeomorphism is said to be {\it periodic}\, if there is $k \in \N^{\ast}$
such that $f^k$ coincides with the identity on a neighborhood of the origin. Clearly periodic diffeomorphisms possess finite orbits.
To prove Theorem~A, we first need to show the following.

\begin{prop}\label{propperiodic}
Let $h$ be an element of a subgroup $G \subseteq \diffn$, where $G$ is as stated in Theorem~A. Then $h$ is
periodic.
\end{prop}

Assuming that Proposition~\ref{propperiodic} holds, then Theorem~A can be derived as follows:

\begin{proof}[Proof of Theorem~A]
We want to prove that $G$ is finite (for example at germs level). So, let us consider the homomorphism
$\sigma : G \rightarrow GL(n, \C)$ assigning to an element $h \in G$ its derivative $D_0 h$ at the origin.
The image $\sigma (G)$ of $G$ is a finitely generated subgroup of $GL(n, \C)$ all of whose elements have finite order.
According to Schur's theorem concerning the affirmative solution of Burnside problem for linear groups, the group
$\sigma (G)$ must be finite, cf. \cite{burnside}. Therefore, to conclude that
$G$ is itself finite, it suffices to check that $\sigma$ is one-to-one or, equivalently, that its kernel is reduced to the identity.
Hence suppose that $h \in G$ lies in the kernel of $\sigma$, i.e. $D_0 h$ coincides with the identity. To show that
$h$ itself coincides with the identity, note that $h$ must be periodic since it has finite orbits, cf. Proposition~\ref{propperiodic}.
Therefore $h$ is conjugate to its linear part at the origin, i.e. it is conjugate to the identity map. Thus $h$ coincides
with the identity on a neighborhood of the origin of $\C^n$. The theorem is proved.
\end{proof}

Before proving Proposition~\ref{propperiodic}, let us make some comments concerning the proof of Theorem~A.
Concerning the case $n=1$, Leau theorem immediately implies that the above considered homomorphism $\sigma$
is one-to-one so that $G$ will be abelian and, indeed, cyclic. This fact does not carry over higher dimensions since,
as already mentioned, every finite group
can be realized as a matrix group and therefore as a pseudogroup of $\Diff$ having finite orbits. Yet, in general, groups obtained in this manner
contain non-trivial elements with non-isolated fixed points. Therefore, if we are dealing with pseudogroups satisfying the conditions of
Theorem~A, the question on whether $G$ is abelian may still be raised. However, even in this restricted setting the group $G$ need not be
abelian. For example, let $G$ be the subgroup of ${{\rm Diff}\, ({\mathbb C^2}, 0)}$
generated by $h_1(x,y) = (e^{\pi i/3} x, e^{2\pi i/3} y)$ and $h_2(x,y) = (y,x)$. Every element of $G$ has finite orbits and possesses a single
fixed point at the origin but the group $G$ is not abelian.

Let us now prove Proposition~\ref{propperiodic}. As already pointed out, the proof amounts to a careful reading of the
argument supplied in \cite{M-M} for the case $n=1$.

\begin{proof}[Proof of Proposition~\ref{propperiodic}]
Let $h$ be a local diffeomorphims in $\diffn$ whose periodic points are isolated unless the corresponding power of $h$
coincides with the identity on a neighborhood of the mentioned point. Let us assume that $h$ is not periodic. To prove the
statement, we are going to show the existence of an open neighborhood $U$ of $0 \in \C^n$ such that the set of points
$x \in U$ with infinite $U_{<h>}$-orbit is uncountable and has the origin as an accumulation point. It will then result that $h$ cannot have finite
orbits, thus proving the proposition.

Let $U$ be an arbitrarily small open neighborhood of $0 \in \C^n$ contained in the domains of definition of $h, \, h^{-1}$. Suppose also
that $h, \, h^{-1}$ are one-to-one on $U$. Consider
$\rho_0 > 0$ such that $D_{\rho_0} \subseteq U$, where $D_{\rho_0}$ stands for the closed ball of radius $\rho_0$ centered at the origin.
Following \cite{M-M}, we define the following sets
\begin{eqnarray*}
{\bf P} & = & \{x \in D_{\rho_0} : \; \mu_{D_{\rho_0}}(x) = \infty , \; \# \calO_{D_{\rho_0}}^{<h>}(x) < \infty\} \, , \\
{\bf F} & = & \{x \in D_{\rho_0} : \; \mu_{D_{\rho_0}}(x) < \infty , \; \# \calO_{D_{\rho_0}}^{<h>}(x) < \infty\} \, , \\
{\bf I} & = & \{x \in D_{\rho_0} : \; \mu_{D_{\rho_0}}(x) = \infty , \; \# \calO_{D_{\rho_0}}^{<h>}(x) = \infty\} \, . \\
\end{eqnarray*}
In other words, ${\bf P}$ is the set of periodic points in $D_{\rho_0}$ for $h$, ${\bf F}$ denotes the set of points leaving $D_{\rho_0}$ after finitely many
iterations and ${\bf I}$ stands for the set of non-periodic points with infinite orbit. Naturally, $D_{\rho_0} = {\bf P} \cup {\bf F}
\cup {\bf I}$ and Lewowicz's lemma implies that
\[
({\bf P} \cup {\bf I}) \cap \partial D_\rho\neq \emptyset \, .
\]
for every $\rho \leq \rho_0$. Thus, at least one between ${\it P}$ and ${\bf I}$ is uncountable. In what follows, the diffeomorphism $h$ is supposed
to be non-periodic. With this assumption, our purpose is to show that ${\bf I}$ must be uncountable.

For $n \geq 0$, let $A_n$ denote the domain of definition of $h^n$ viewed as an element of the pseudogroup {\it generated on $D_{\rho_0}$}. Clearly
$A_{n+1} \subseteq A_n$. Next, let $C_n$ be the connected (compact) component of $A_n$ containing the origin and set
\[
C = \bigcap_{n \in \N} C_n \, .
\]
Note that $C$ is the intersection of a decreasing sequence of compact connected sets. Therefore $C$ is non-empty and connected.

\vspace{0.1cm}

\noindent {\it Claim}: Without loss of generality, the set $C$ can be supposed countable.

\begin{proof}[Proof of the Claim]
Suppose that $C$ is uncountable.
The reader is reminded that our aim is to conclude that ${\bf I}$ is uncountable provided that $h$ is not periodic. Therefore
we suppose for a contradiction that ${\bf I}$ is countable.
Since ${\bf I}$ is countable so is ${\bf I} \cap C$.
Consider now $C \cap {\bf P}$ and note that this intersection must be uncountable, since $C \subset {\bf P} \cup {\bf I}$. Let
\[
C \cap {\bf P} = \bigcup_{n\in\N} P_n \, ,
\]
where $P_n$ is the set of points $x \in C \cap {\bf P}$ of period $n$. Note that there exists a certain $n_0 \in \N$ such that
$P_{n_0}$ is infinite, otherwise all of the $P_n$ would be finite and $C \cap {\bf P}$ would be countable. Being infinite,
$P_{n_0}$ has a non-trivial accumulation point $p$ in $C_{n_0}$. The map $h^{n_0}$ is holomorphic on an open neighborhood
of $C_{n_0}$ and it is the identity on $P_{n_0} \cap C_{n_0}$. Since $p$ is not an isolated fixed point of $h^{n_0}$, it
follows that $h^{n_0}$ coincides with the identity map on $C_{n_0}$, i.e. on the connected component of the domain of
definition of $h^{n_0}$ that contains the origin. This contradicts the assumption of non-periodicity of $h$ (modulo reducing the
neighborhood of the origin). Hence ${\bf I}$
is uncountable as desired.
\end{proof}

In view of the preceding, in the sequel $C$ will be supposed to consist of countably many points. The purpose is still to conclude
that the set ${\bf I}$ is uncountable. Since $C$ is
connected, it follows that $C$ is reduced to the origin. Then, for every $\rho < \rho_0$, we have $C \cap \partial D_\rho =
\emptyset$. Now note that, for a fixed $\rho > 0$, the sets
\[
C_1 \cap \partial D_\rho, \, \,  (C_1 \cap C_2) \cap \partial D_\rho, \, \,  (C_1 \cap C_2 \cap C_3)\cap \partial D_\rho, \, \,  \ldots
\]
form a decreasing sequence of compact sets. Hence the intersection $\bigcap_{n\in\N} C_n \cap \partial D_\rho$ is non-empty,
unless there exists $n_0 \in \N$, such that $C_{n_0} \cap \partial D_\rho = \emptyset$. The latter case must occur since
$C \cap \partial D_\rho = \emptyset$. However, the value of $n_0$
for which the mentioned intersection becomes empty may depend on $\rho$.

Fix $\rho > 0$ and let $n_0$ be as above. Let $K$ be a compact connected neighborhood of $C_{n_0}$ that does not intersect
the other connected components of $A_{n_0}$, if they exist. The set $K$ can be chosen so that $\partial K \cap A_{n_0} = \emptyset$.
The inclusion $A_{n+1} \subseteq A_n$ guarantees that $\partial K$ does not intersect $A_n$, for every $n \geq n_0$.
Therefore
\[
\partial K \cap {\bf P} = \emptyset \, .
\]
In fact, if there were a periodic point $x$ of $D_{\rho_0}$ on $\partial K$, then $x$ would belong to every set $A_n$. In
particular, it would belong to $A_{n_0}$, hence leading to a contradiction. Nonetheless, Lewowicz's lemma guarantees the existence
of a point $x$ on the boundary $\partial K$ of $K$ such that the number of iterations in $K$ is infinite, i.e. such that
$\mu_K (x) = \infty$. Since $K \subseteq D_{\rho} \subseteq D_{\rho_0}$, it follows that
\[
\partial K \cap {\bf I} \ne \emptyset \, .
\]
By construction, it is clear that a compact set $K$ satisfying the properties above is not unique. Indeed, for $K$ as above, denote
by $K_{\varepsilon}$ the compact connected neighborhood of $K$ whose boundary has distance to
$\partial K$ equal to $\varepsilon$. Then, there exists $\varepsilon_0 > 0$ such that $K_{\varepsilon}$ satisfies the
same properties as $K$ for every $0 \leq \varepsilon \leq \varepsilon_0$ with respect to $A_{n_0}$. In particular,
\[
\partial K_{\varepsilon} \cap {\bf I} \ne \emptyset
\]
for all $0 \leq \varepsilon \leq \varepsilon_0$. Therefore ${\bf I}$ must be uncountable. Finally, it remains to prove
that $0 \in \C^n$ is an accumulation point of ${\bf I}$. This is, however, a simple consequence of the fact that a compact set $K
\subseteq D_{\rho}$ as above can be considered for all $\rho > 0$. This completes the proof of Proposition~\ref{propperiodic}.
\end{proof}

\begin{obs}
{\rm Concerning Proposition~\ref{propperiodic} in dimension~$2$ and in the case where $h$ is tangent to the identity, the use made
in \cite{scardua} of Abate's theorem \cite{abate} has an advantage compared to Proposition~\ref{propperiodic}, namely:
Abate's theorem shows that it suffices to check that the origin is an
isolated fixed point of $h$ itself whereas, in more general cases, Proposition~\ref{propperiodic} requires all non-trivial powers of $h$ to
have only isolated fixed points.}
\end{obs}

We can now move on to prove Theorem~B. The proof of this theorem follows from
the combination of our Theorem~A with the results in \cite{EI} or in \cite{helena}.

To begin with, let $\fol$ be a singular foliation on $(\C^n, 0)$ and let $X$ be a representative
of $\fol$, i.e. $X$ is a holomorphic vector field tangent to $\fol$ and whose singular set has codimension at
least~$2$. Suppose that the origin is a singular point of $\fol$ and denote by $\dl_1, \ldots, \dl_n$ the corresponding eigenvalues of $DX$ at the origin.
Assume the following holds:
\begin{enumerate}
\item $\fol$ has an isolated singularity the origin.
\item The singularity of $\fol$ is of Siegel type.
\item The eigenvalues $\dl_1, \ldots, \dl_n$ are all different from zero and
there exists a straight line through the origin, in the complex plane, separating $\dl_1$ from the remainder eigenvalues.
\item Up to a change of coordinates, $X = \sum_{i=1}^n \dl_ix_i(1+f_i(x)) \partial /\partial x_i$, where $x=(x_1,\ldots,x_n)$ and $f_i(0)=0$ for all $i$
\end{enumerate}
The fourth condition above amounts to assuming the existence of $n$~invariant hyperplanes
through the origin. This condition as well as condition~(3) are always verified when $n=3$ provided that the singular point is of
{\it strict Siegel type}\, cf. \cite{C}. Also, recall that the singular point is said to be
of strict Siegel type if $0 \in \C$ is contained in the interior of the convex hull of $\{\dl_1,\ldots,\dl_n\}$.

Next we shall need Theorem~\ref{TMMhigher} below. This theorem generalizes to higher dimensions a unpublished
result of Mattei which, in turn, improved on an earlier version appearing in \cite{M-M}.

\begin{theorem}\label{TMMhigher}
{\rm ({\bf [EI],  [Re]})}
Let $X$ and $Y$ be two vector fields satisfying conditions (1), (2), (3) and~(4) above.
Denote by $h^X$ (resp. $h^Y$) the holonomy of $X$ (resp. $Y$) relatively to
the separatrix of $X$ (resp. $Y$) tangent to the eigenspace associated to the
first eigenvalue. Then $h^X$ and $h^Y$ are analytically conjugate if and only
if the foliations associated to $X$ and $Y$ are analytically equivalent.
\end{theorem}

The proof of Theorem~\ref{TMMhigher} can be found in either \cite{EI} or \cite{helena}, a particularly detailed exposition
can be found in \cite{monograph}.
With this theorem in hand, the proof of Theorem~B goes as follows.

\begin{proof}[Proof of Theorem~B]
Let $X$ be as in the statement and denote by $\fol$ the foliation associated to $X$. Let $x_1$ be the invariant axis corresponding to the eigenvalue
$\dl_1$ as in item~(3) above. Consider also the local holonomy map $h$ relative to this invariant axis and to the foliation $\fol$. The map $h$ is defined
on a suitable local section and it can also be identified to a local diffeomorphism fixing the origin of $\C^{n-1}$. By assumption, all iterates of $h$ have isolated
fixed points. Therefore Theorem~A implies that the local orbits of $h$ are finite if and only if $h$ is periodic. Naturally, we may
assume this to be the case. Let then $N$ be the {\it period}\, of $h$, namely the smallest strictly positive integer for which $h^N$ coincides with the identity
on a neighborhood of the origin of $\C^{n-1}$ (with the above mentioned identifications). Denote also by $T$ the derivative of $h$ at the origin, which
is itself identified to a linear transformation of $\C^{n-1}$. The fact that $h$ is periodic of period~$N$ ensures that $T$ is also periodic with the same
period~$N$. In fact, $h$ and $T$ are analytically conjugate as already mentioned (i.e. $h$ is linearizable).
Moreover, $T$ is the holonomy map with respect to the axis $x_1$ associated to the foliation $\fol_Z$ induced by the linear vector field
$$
Z = \sum_{i=1}^n \dl_ix_i \partial /\partial x_i \, .
$$
It follows from Theorem~\ref{TMMhigher} that the foliation $\fol$ is analytically equivalent to the foliation $\fol_Z$. However, for $\fol_Z$
(i.e. a foliation induced by a linear diagonal vector field), it is immediate to check that complete integrability is equivalent to the periodic
character of the holonomy map $T$. Since $\fol$ and $\fol_Z$ are analytically equivalent, we conclude from what precedes that the condition
of having a local holonomy $h$ with finite orbits forces $\fol$ to be completely integrable. The converse is clear, since having $\fol$
completely integrable ensures at once that the holonomy map $h$ must be periodic. This finishes the proof of Theorem~B.
\end{proof}


\section{Examples of dynamics and the complement to Theorem~B}

This section contains some interesting examples of diffeomorphisms of $(\C^2,0)$ tangent to the identity and possessing
``special'' local dynamics along with examples of foliations where these diffeomorphisms are realized as holonomy
maps. Among these examples, the vector field mentioned in the ``complement to Theorem~B'' will be discussed in detail.

Naturally we are, in particular, interested in examples of diffeomorphisms tangent to the identity at $(0,0) \in \C^2$ and having finite orbits.
According to Theorem~A, none of these diffeomorphisms may have an isolated fixed point at the origin. Recall also that
diffeomorphisms tangent to the identity are realized as time-one maps of {\it formal vector fields}. This formal vector field
is unique and it is referred to as the {\it infinitesimal generator}\, of the diffeomorphism in question, cf. Section~3. In particular,
it is natural to look for examples among diffeomorphisms
that are time-one maps of actual holomorphic vector fields, or at least that leafwise preserve some singular holomorphic foliation. Here as usual,
it is convenient to distinguish between a vector field $X$ and its associated foliation $\fol$ obtained by eliminating non-trivial common factors among the
components of $X$: since our diffeomorphisms do not have isolated fixed points, their infinitesimal generators will not have isolated
singularities either.

\subsection{Local diffeomorphisms}

Recall that $\diffCtwo$ denotes the normal subgroup of $\Diffgentwo$ consisting of diffeomorphisms tangent to the identity.
Let us begin with some examples of diffeomorphisms in $\diffCtwo$ with interesting
local dynamics, including examples possessing finite orbits.
The simplest case where $F(x,y) = (x + f(y) , y)$, with $f(0) =f'(0) =0$, can be set aside in what follows. Note that
the foliation associated to the infinitesimal generator of $F$ is regular in this case. Examples of diffeomorphisms whose
infinitesimal generator provides a singular foliation can also be produced by successively blowing-up $F$.
Nonetheless other examples of diffeomorphisms associated to linear foliations are described below.

\vspace{0.1cm}

\noindent {\bf Example 1}: Linear vector fields.

Consider the vector field $Y$ given by $Y = x \partial /\partial x -  \lambda y \partial /\partial y$
where $\lambda =n/m$ with $m,n \in \N^{\ast}$. The foliation associated to $Y$ will be denoted by $\fol$ and it should be noted
that the holomorphic function $(x,y) \mapsto x^n y^m$ is a first integral for $\fol$. Let $\phi_Y$ denote the time-one map
induced by $Y$. The local dynamics of $\phi_Y$ can easily be described as follows. The vector field $Y$ can be projected on the axis
$\{ y=0\}$ as the vector field $x \partial /\partial x$. Therefore the (real) integral curves of $Y$ coincide
with the lifts in the corresponding leaves of $\fol$ of the (real) trajectories of $x \partial /\partial x$ on $\{ y=0\}$.
The latter trajectories are radial lines being emanated from $0 \in \{ y=0\} \simeq \C$ so that
the local dynamics of $\phi_Y$ restricted to $\{ y=0\}$ is such that, whenever $x_0 \neq 0$, the sequence $\{\phi_Y^n (x_0) \}$ marches off a
uniform neighborhood of $0 \in  \{ y=0\} \simeq \C$ as $n \rightarrow \infty$ and it converges to $0 \in  \{ y=0\} \simeq \C$
as $n \rightarrow -\infty$. Consider now the orbit of
a point $(x_0, y_0)$, $x_0y_0 \neq 0$, by $\phi_Y$. Since this is simply the lift in the leaf of $\fol$ through $(x_0, y_0)$ of the dynamics
of $x_0 \in \{ y=0\} \simeq \C$, it follows that $\phi_Y^n (x_0, y_0)$ leaves a fixed
neighborhood of $(0,0) \in \C^2$ since the first coordinate increases to uniformly large values provided that $n \rightarrow \infty$.
Similarly, when $n \rightarrow -\infty$, the first coordinate of $\phi_Y^n (x_0, y_0)$ must converge to
{\it zero}\, so that the second coordinate becomes ``large'' due to the first integral $x^ny^m$. Thus, fixed a (small) neighborhood
$U$ of $(0,0) \in \C^2$, every orbit of $\phi_Y$ that is not contained in $\{ x=0\} \cup \{ y=0\}$ is bound to intersect $U$ at finitely many points.

Clearly the time-one map induced by $Y$ is not tangent to the identity. However, examples of time-one maps tangent to the identity and satisfying
the desired conditions can be obtained, for example, by considering the vector field $X = x^ny^m Y$ and taking the time-one map $\phi_X$ induced by $X$.
Clearly the linear part of $X$ at $(0,0)$ equals zero so that $\phi_X$ must be tangent to the identity. Furthermore, the multiplicative factor $x^ny^m$
annihilates the dynamics of $\phi_X$ over the coordinate axes so that only the orbits of points $(x_0, y_0)$, with $x_0y_0 \neq 0$, have to be considered.
The leaf of $\fol$ through $(x_0, y_0)$ will be denoted by $L$. Also let $c \in \C$ be the value of $x^ny^m$ on $L$.
The restriction of $X$ to $L$ is nothing but the restriction of $Y$ to $L$ multiplied by the scalar
$c \in \C$. Therefore the real orbits of $X$ in $L$ coincide with the lift to $L$ of the real orbits of the vector field
$cx  \partial /\partial x$ defined on $\{ y=0\}$.  The geometric nature of the orbits of $cx \partial /\partial x$
depends on the argument of $c \in \C$, i.e. setting $c= \vert c \vert
e^{2\pi i \alpha}$, this geometry depends on $\alpha \in [0, 2\pi)$. If $\alpha = \pi/2$, then the orbits of $cx \partial /\partial x$ are contained in circles about the origin.
After finitely many tours, these circles lift into the corresponding leaf (i.e. the leaf on which $xy$ equals~$c$) as closed paths invariant by $\phi_X$. In addition
for a ``generic'' choice of $c$ satisfying $\alpha = \pi/2$, the resulting time-one map restricted to the corresponding invariant path
will be conjugate to an irrational rotation. Thus
$\Phi_X$ does not have finite orbits.

Let us now briefly discuss the slightly more general case where $X =x^a y^b Y$ with $a,b \in \N^{\ast}$. Setting $d = am-bn$, we can suppose without
loss of generality that $d\geq 1$. Next, by considering the system
$$
\begin{cases}
\frac{dx}{dt} = mx^{a+1} y^b \\
\frac{dy}{dt} = -nx^a y^{b+1} \, ,
\end{cases}
$$
we conclude that $dy/dx = -ny/mx$ so that $y = c x^{-n/m}$ in ramified coordinates. In turn, this
yields $dx/dt = c^b mx^{1 + d/m} \partial /\partial x$. Since $d \geq 1$ by construction, the orbits of the latter
vector field defines the well-known ``petals'' associated to Leau flower in the case of periodic linear
part, cf. \cite{carlerson}. For example, setting $m=1$ to simplify,
the orbits of the vector field $x^{1+d} \partial /\partial x$ consists of $d+1$ ``petals'' in non-ramified coordinates.

In any event, the sequence of points in $\{ y=0\}$ consisting of the first coordinates of the
full orbits of $\phi_X$ either marches straight off a neighborhood of $0 \in \{ y=0\} \simeq \C$ or it converges to $(0,0)$.
Converging to $(0,0)$ will force the second coordinates of the points in the $\phi_X$-orbit to increase uniformly so that
the orbit in question must leave a fixed neighborhood of $(0,0) \in \C^2$. Summarizing, we conclude:

\noindent {\it Claim 1}. Fixed a neighborhood $U$ of $(0,0) \in \C^2$ and given $p = (x_0, y_0)$, $x_0y_0 \neq 0$, the set
$$
U \cap \left\{ \bigcup_{n=-\infty}^{\infty} \phi_X^n (p) \right\}
$$
is finite.

\vspace{0.1cm}

\noindent {\bf Example 2}: Diffeomorphisms leaving the function $(x,y) \mapsto xy$ invariant.

Let us see here two cases similar to Example~1 that can easily be realized as the holonomy of foliations as in Theorem~B. First, let $F \in \diffCtwo$ be given by
\begin{equation}
F (x,y) = [ x(1+xy f(xy)), \, y (1 +xy f(xy))^{-1}] \, , \label{forholonomy1}
\end{equation}
where $f (z)$ is a holomorphic function defined about $0 \in \C$ and satisfying $f(0) \neq 0$.
Note that $F$ leaves the function $(x,y) \mapsto xy$ invariant since the product of its first and second components
equals~$xy$.

Next, consider an initial point $(x_0,y_0)$ with $x_0y_0 =C \neq 0$. The orbit of $(x_0,y_0)$
under $F$ is hence contained in the curve defined by $\{ xy = C\}$. However, for a point $(\tilde{x} , \tilde{y})$ lying in $\{ xy = C\}$, the value of
$F (\tilde{x} , \tilde{y})$ takes on the form
$$
F (\tilde{x} , \tilde{y}) = [\tilde{x} (1 + C  f(C)) \, , \,  \tilde{y} (1 + C  f(C))^{-1} ] \, .
$$
In particular, those values of $C$ for which $\vert 1 + C  f(C) \vert =1$ give rise to a rotation in the first coordinate. Therefore the lifts
of these circles in the corresponding leaves are loops. Besides for a generic choice of $C$ satisfying $\vert 1 + C  f(C) \vert =1$ the dynamics
induced on one of these invariant loops is conjugate to an irrational rotation so that $F$ does not have finite orbits.

Consider now the local diffeomorphism $H$ which is given by
\begin{equation}
H (x,y) = [ x(1+x^2y f(x^2y)), \, y (1 +x^2y f(x^2y))^{-1}] \, , \label{forholonomy2}
\end{equation}
where $f$ is as above.
It follows again that $H$ preserves the function $(x,y) \mapsto xy$. Our purpose is to show that, unlike $F$, $H$ has
finite orbits. For this, fix again an initial point $(x_0,y_0)$ with $x_0y_0 =C \neq 0$ so that the orbit of $(x_0,y_0)$
under $H$ is contained in the curve $\{ xy = C\}$. Next note that, if $(\tilde{x} , \tilde{y})$ lies in $\{ xy = C\}$, we have
$$
H (\tilde{x} , \tilde{y}) = [\tilde{x}(1 + \tilde{x} C f( \tilde{x} C) \, , \,  \tilde{y} (1 + \tilde{x}C f(\tilde{x} C))^{-1} ] \, .
$$
The dynamics of the first component of $H$ behaves again as the Leau flower. Therefore, by resorting to an argument totally
analogous to the one employed in Example~1 for $X =x^a y^b Y$ with $d = am-bn \neq 0$, we conclude that all the orbits
of $H$ are finite as desired.

\vspace{0.1cm}

\noindent {\bf Example 3}: Resonant vector fields.

This example consists of a diffeomorphism with finite orbits that is associated to a non-linear vector field. Indeed, the previous examples
of diffeomorphisms preserved foliations admitting non-constant holomorphic first integrals. To obtain a non-linear example
possessing no holomorphic first integral, consider a resonant
vector field $Y$ about $(0,0) \in \C^2$ as in \cite{M-RamisENS}, the simplest example being
\begin{equation}
Y = x \partial /\partial x - y(1 + xy) \partial /\partial y \, . \label{bernoulli}
\end{equation}
The vector field $Y$ is not linearizable and, indeed, possesses only constant holomorphic first integrals. Next let
$X =xyY$ and denote by $\phi_X$ the time-one map induced by $X$. We shall prove the following:

\vspace{0.1cm}

\noindent {\it Claim 2}. Fixed a neighborhood $U$ of $(0,0) \in \C^2$ and given $p = (x_0, y_0)$, $x_0y_0 \neq 0$, the set
$$
U \cap \left\{ \bigcup_{n=-\infty}^{\infty} \phi_X^n (p) \right\}
$$
is finite.

\begin{proof} Considering the vector field $Y$ in~(\ref{bernoulli}), it follows that the resulting equation for $dy/dt$ is a classical
Bernoulli equation. This equation can explicitly be integrated to yield the solutions
$$
x(t) = x_0 e^t \; \; \, {\rm and} \; \; \, y(t) = \frac{y_0}{e^t (1 + x_0y_0 t)}
$$
corresponding to the integral curves of $Y$. In particular, it follows that
$$
xy = \frac{x_0y_0}{1 +x_0y_0 t} \, .
$$
Inputting the last equation into the equation for $dx/dt$ arising from the vector field $X$, it follows that $x (t) = x_0 (1 +x_0y_0 t)$.
Therefore, setting $t=1$, the map $\phi_X$ must be given by
\begin{equation}
\phi_X (x_0, y_0) = [ x_0 +x_0^2y_0 , y_0 (1 +x_0y_0)^{-2} ] \, . \label{diffeosaddlenode}
\end{equation}
Thus, whenever the iteration $\phi_X^n$ is defined, its first component equals~$x_0 + n x_0^2y_0$. Since, by assumption $x_0^2y_0 \neq 0$,
this component behaves as a non-trivial translation whose orbits necessarily march off a neighborhood of $0 \in \C$.
The claim follows at once.
\end{proof}

\vspace{0.1cm}

\noindent {\bf Example 4}: Dynamics on a pencil of elliptic curves: Invariant sets and no parabolic domain.

Let us now discuss a more elaborate case where the vector field $Y$ itself has {\it zero}\, linear part at the origin. This means that the
time-one map induced by $Y$ is already tangent to the identity. This contrasts with the previous examples for which we needed to
use vector fields with non-isolated singularities to obtain diffeomorphisms tangent to the identity.
The choice of $Y$ to be made below is such that the
leaves of the associated foliation $\fol$ have {\it more topology}\, in the sense that they are punctured elliptic curves. The influence of this
topology will significantly change the nature of the results obtained.

Consider the vector field $Y = x(x-2y) \partial /\partial x + y(y-2x) \partial /\partial y$ admitting $xy(x-y)$ as holomorphic first integral.
In particular, $Y$ possesses exactly three separatrices, namely the coordinate axes and the line $\{ x=y\}$. Also consider the vector field
$X = xy (x-y)  Y$, since vector fields vanishing identically over all its invariant curves through $(0,0) \in \C^2$ are needed to yield diffeomorphisms
tangent to the identity and having locally closed orbits. Note that the blow-up of $Y$ yields a foliation having three
linearizable singularities on the corresponding exceptional divisor. About each of these singularities, the vector field $X$
is as in Example~1, with $d=am-bn \neq 0$. In particular, the corresponding time-one map has finite orbits about each of these three singular points.
However, it will be seen that, globally, this time-one map does not have finite orbits due to the influence of the topology
of the leaves.

To describe the dynamics of $X$, let us first consider the case of $Y$. Note that the orbits of $Y$ are the fibers of the pencil induced on
$\C P(2)$ by the first integral $xy(x-y)$. Except for the tree invariant lines determined by $xy(x-y) =0$,
these fibers are smooth (projective) elliptic curves which are pairwise isomorphic since they are permuted
by the flow of radial vector field $x \partial /\partial x
+ y \partial /\partial y$. The pencil has $3$ singular points at the ``line at infinity''
$\Delta_{\infty}$ corresponding to the invariant directions $\{ x=0\}$, $\{y=0\}$ and $\{ x=y\}$ and all the
elliptic curves pass through each of these singular points.
Whereas $X$ has poles on $\Delta_{\infty}$,
a straightforward change of coordinates shows that {\it the restriction $X_L$ of
$X$ to every elliptic curve $L$}\, extends holomorphically to a vector field defined on all of $L$. Thus $X_L$ is a constant vector field on $L$. A slightly more
detailed analysis, cf. \cite{ghys-r} page 1150, shows that $X_L$ does not depend on $L$ in the sense that, if
$\sigma : L \rightarrow  L'$ is a holomorphic diffeomorphism between two elliptic curves as above, then $h^{\ast} X_{L'} = X_L$.

With the preceding information in hand, fix $L$ as above and consider now the Weierstrass representation of $L$ as the parallelogram $\mathcal{P}$ in $\C$
defined by $1$ and $\tau \in C^{\ast}$ (actually $\tau = e^{\pi \sqrt{-1}/3}$) identified to $L$ through the corresponding
Weierstrass $\wp$-function. Modulo multiplying $Y$ by a constant, the vector field induced by
$Y_L$ on $\mathcal{P}$ is simply $\partial /\partial T$, where $T$ is a global coordinate on $\C$.

Let us now fix a neighborhood $U \subset \C^2$ of the origin and consider (a connected component of) the intersection
$L_U$ of $L$ and $U$, where $L$ is an elliptic curve as above. The parameterization $\wp$ yields an identification of $L_U$ with
$\mathcal{P} \setminus (B_1 \cup B_2 \cup B_3)$ the three-times holed parallelogram $\mathcal{P}$, where the holes $B_1\, , B_2, \, B_3$ are in
natural correspondence with the invariant lines $\{ x=0\}$, $\{y=0\}$ and $\{ x=y\}$. Furthermore, when the leaf $L$ varies,
$\mathcal{P} \setminus (B_1 \cup B_2 \cup B_3)$ does not change since the leaves are pairwise isomorphic.
Nonetheless, since the first integral $xy(x-y)$ varies,
the restriction of $X$ to $L$ viewed in $\mathcal{P} \setminus (B_1 \cup B_2 \cup B_3)$
becomes a constant times the vector field $\partial /\partial T$. The value of this constant depends on the leaf and goes to
zero if the leaf is close to any of the three invariant lines.

Let us now consider the dynamics of $Z = \partial /\partial T$ on $\mathcal{P} \setminus (B_1 \cup B_2 \cup B_3)$ for real time. This dynamics
consists of horizontal (real) lines in $\mathcal{P} \setminus (B_1 \cup B_2 \cup B_3)$: those lines who intersect $B_1 \cup B_2 \cup B_3$ leave
the  neighborhood $U$ (and therefore are ``ended'' from the local point of view). The remaining lines give rise to periodic orbits. In particular,
if $f$ is the local diffeomorphism induced as time-one map of $X$, $f$ possesses fully invariant sets contained in $U$ and away from the
invariant lines $\{ x=0\}$, $\{y=0\}$ and $\{ x=y\}$.

Consider now the vector field $e^{2\pi i \alpha}Y$ where $\alpha \in \C$ is a constant. Taking advantage of the previous construction,
the new vector field $e^{2\pi i \alpha}Z$ viewed in $\mathcal{P} \setminus (B_1 \cup B_2 \cup B_3)$ coincides with
$e^{2\pi i \alpha} \partial /\partial T$. The value of $\alpha$
is chosen with irrational real part, so that the slope of the ``real direction'' of $e^{2\pi i \alpha} \partial /\partial T$ is irrational.
Therefore the real flow of
$e^{2\pi i \alpha}Z$ on $\mathcal{P}$ is a linear irrational flow all of whose orbits are dense. In particular, its restriction to
$\mathcal{P}  \setminus (B_1 \cup B_2 \cup B_3)$ is such that every orbit will eventually intersect the holes $B_1 \cup B_2 \cup B_3$ and then
leave the neighborhood $U$. Denoting by $\phi$ the time-one map induced by $e^{2\pi i \alpha}X$, it follows that the orbit of
every point in $L_U$ will intersect $L_U$ in finitely many points only (unless it ``jumps over the holes'').
In any event, the previous constructed diffeomorphisms do not have finite orbits in $U$ due to
the fact that their dynamics restricted to the invariant lines $\{ x=0\}$, $\{y=0\}$ and $\{ x=y\}$ consists of Leau flowers. As previously done, we
may then think of multiplying $cX$ by its first integral $xy(x-y)$ so as to annihilate
the dynamics over these invariant lines. This strategy, however, does not work in the present case due to the following:

\vspace{0.1cm}

\noindent {\it Claim 4}. There are leaves $L_U$ of the restriction of $cxy(x-y) X$ to $U$ containing fully invariant
sets by the dynamics of this vector field in real time.

\noindent {\it Proof}. Keeping the preceding notations, the restriction of $cxy(x-y) X$ to $\mathcal{P} \setminus (B_1 \cup B_2 \cup B_3)$ is given
by $cc_L \partial /\partial T$ where $c_L$ is a constant depending on the leaf $L$, more precisely $c_L$ equals the value of
the first integral $xy(x-y)$ over $L$. Therefore, regardless of the chosen value of $c$, there will always exist leaves $L$ over which the real
dynamics of $cxy(x-y) X$ is constituted by periodic orbits, some of them avoiding the holes $B_1 \cup B_2 \cup B_3$. The claim follows.\qed

It follows from this claim that the diffeomorphism $\phi_X$ induced as time-one map of $X = xy(x-y)Y$ has closed invariant sets not containing the
origin whereas it does not possess any parabolic domain in the sense of \cite{hakim}, \cite{abate}.


\subsection{Singular foliations and holonomy}

This paragraph contains simple examples of foliations giving rise to holonomy maps with properties similar to
the above discussed cases. These examples include the foliation introduced in the ``complement to Theorem~B''.

Let us begin by pointing out a simple observation that shows that every element of $\diffCtwo$ can be realized as a local
holonomy map for some foliation. Indeed, consider
a singular foliation $\fol$ on $(\C^3,0)$ admitting a separatrix $S$ through the origin and denote by
$h$ the holonomy map associated to $\fol$, with respect to $S$. Assume that the foliation is locally given by
the vector field $A(x,y,z) \partial /\partial x + B(x,y,z) \partial /\partial y + C(x,y,z) \partial /\partial
z$. Assume furthermore that the separatrix $S$ is given, in the same coordinates, by $\{x=0, y=0\}$. Setting $z = e^{2\pi it}$, the corresponding
holonomy map can be viewed as the time-one map associated to the differential equation
\[
\begin{cases}
\frac{dx}{dt} & = \frac{dx}{dz} \frac{dz}{dt} = 2\pi i e^{2\pi i t} \frac{A(x,y,e^{2 \pi i t})}{C(x,y,e^{2\pi it})}\\
\frac{dy}{dt} & = \frac{dx}{dz} \frac{dz}{dt} = 2\pi i e^{2\pi i t} \frac{B(x,y,e^{2 \pi i t})}{C(x,y,e^{2\pi it})}
\end{cases} \, .
\]
In the particular case where $A, \, B$ do not depend on $z$ and $C$ is reduced to $C(x,y,z) = z$, the holonomy
map of $\fol$ with respect to $S$ reduces to the time-one map induced by a vector field on $(\C^2, 0)$, namely by
the vector field
\[
2\pi i \left[ A(x,y) \frac{\partial}{\partial x} + B(x,y) \frac{\partial}{\partial y} \right] \, .
\]

Consider, for example, the local diffeomorphism $h$ introduced in Example~3 and recall that $h$ is the time-one map
induced by the vector field $Y = xy \left[x \partial /\partial x - y(1 + xy) \partial /\partial y \right]$.
To find a vector field on $(\C^3,0)$ whose
foliation has $h$ as holonomy map, it suffices to take $Y$ and ``join" the term $\frac{1}{2\pi i}z \partial
/\partial z$. Then the holonomy of the foliation associated to the vector field
\[
X = xy \left[x \frac{\partial }{\partial x} - y(1 + xy) \frac{\partial }{\partial y} \right] + \frac{1}{2pi i} z \frac{\partial }{\partial z} \, ,
\]
with respect to the $z$-axis is nothing but $h$ itself.

Note that the vector field $X$ above corresponds to a saddle-node vector field of codimension~$2$. This is equivalent to say that
its linear part admits exactly two eigenvalues equal to zero and a non-zero eigenvalue associated to the direction of the separatrix
$\{ x=y=0\}$. The fact that the holonomy of $\{ x=y=0\}$ has finite orbits is a phenomenon without analogue for saddle-nodes in
dimension~$2$.

Let us now provide three examples of foliations on $(\C^3,0)$ possessing only eigenvalues different from {\it zero}, as in the case of
Theorem~B.

\noindent {\bf Example 5}: Let $\fol$ denote the foliation associated to the vector field
\[
X = x(1 + xyz^2) \frac{\partial }{\partial x} + y(1 - xyz^2) \frac{\partial }{\partial y} - z \frac{\partial }{\partial z} \, .
\]
The $z$-axis corresponds to one of the separatrices of $\fol$. Taking $z = e^{2\pi it}$, it follows that the holonomy map $h$
associated to $\fol$, with respect to the $z$-axis, is given by the time-one map associated to the vector field
\begin{equation}\label{eqholonomy}
\begin{cases}
\frac{dx}{dt} & = \frac{dx}{dz} \frac{dz}{dt} = -2\pi i x(1 + e^{4\pi it}xy)\\
\frac{dy}{dt} & = \frac{dx}{dz} \frac{dz}{dt} = -2\pi i y(1 - e^{4\pi it}xy)
\end{cases} \, .
\end{equation}
To solve this system of differential equations, we should consider the series expansion of $(x(t),y(t))$ in terms of the
initial condition. More precisely, if $(x(0), y(0)) = (x_0, y_0)$, then we should let $x(t) = \sum a_{ij}(t) x_0^i y_0^j$
and $y(t) = \sum b_{ij}(t) x_0^i y_0^j$. Clearly $a_{10}(0) = b_{01}(0) = 1$ and $a_{ij} = b_{ij}(0) = 0$ in the other
cases. Substituting the series expansion of $x(t)$ and $y(t)$ on~(\ref{eqholonomy}) and comparing the same powers on the
initial conditions, it can be said that the system~(\ref{eqholonomy}) induces an infinite number of differential equations
involving the functions $a_{ij}, \, b_{ij}$ and their derivatives. Each one of the differential equation takes the form
\begin{eqnarray*}
a_{ij}^{\prime}(t) & = & -2\pi i  \left[ a_{ij}(t) + \sum e^{4\pi it} a_{p_1q_1}(t) a_{p_2q_2} (t) b_{p_3q_3}(t) \right] \\
b_{ij}^{\prime}(t) & = & -2\pi i  \left[ b_{ij}(t) + \sum e^{4\pi it} a_{p_1q_1}(t) b_{p_2q_2} (t) b_{p_3q_3}(t) \right]
\end{eqnarray*}
where $p_1 + p_2 +p_3 =i$ and $q_1 + q_2 + q_3 = j$. In particular, the terms on the sum in the right hand side of the equation
above involves only coefficients of the monomials $x_0^p y_0^q$ of degree less then $i + j$ and such that $p \leq i$
and $q \leq j$. Computing this holonomy map becomes much easier with the following lemma:
\begin{lemma}
\label{preservingxy}
The holonomy map $h$ preserves the function $(x,y) \mapsto xy$.
\end{lemma}

\begin{proof}
To check that the level sets of $(x,y) \mapsto xy$ are preserved by $h$, consider the derivative of the product $x(t) y(t)$ with respect
to $t$. This gives us
\begin{align*}
\frac{d}{dt} (xy) &= \frac{dx}{dt}  y + x  \frac{dy}{dt} \\
&= - \left[ 2\pi i x(1 + e^{4\pi it}xy) \right] y - x \left[ 2\pi i y(1 - e^{4\pi it}xy) \right] \\
&= -4\pi i xy \, .
\end{align*}
Thus, by integrating the previous differential equation with respect to the product $xy$, we obtain
\[
(xy)(t) = x_0y_0 e^{-4 \pi i t} \, .
\]
Since the holonomy map corresponds to the time-one map of the system of differential equations~(\ref{eqholonomy}) and
since $e^{-4 \pi i t} = 1$ for all $t \in \Z$, it follows that the orbits of $h$ are contained in the level sets of $(x,y) \mapsto xy$ as desired.
\end{proof}

Lemma~\ref{preservingxy} implies that it suffices to determine the first coordinate of $h$. By recovering the preceding non-autonomous system
of differential equations, a simple induction argument on the value of $i+j$ shows that $h$ has the form
\begin{equation}\label{eqexpressionH}
h(x,y) = (x(1 + xy f(xy)), y(1 + xy f(xy))^{-1}) \, ,
\end{equation}
where $f$ represents a holomorphic function of one complex variable such that $f(0) = 2\pi i$ (the expression for the second coordinate
of $h$ is obtained from the first coordinate by means of Lemma~\ref{preservingxy}). The resulting diffeomorphism $h$ is clearly non-periodic
but does have invariant sets giving by ``circles''. Besides on some of these invariant ``circles'' the dynamics is conjugate to an irrational
rotation, cf. Example~2.

\vspace{0.1cm}

\noindent {\bf Example 6}: {\sc Complement to Theorem B}.

Let $\fol$ denote the foliation associated to the vector field
\[
X = x(1 + x^2yz^3) \frac{\partial }{\partial x} + y(1 - x^2yz^3) \frac{\partial }{\partial y} - z \frac{\partial }{\partial z} \, .
\]
Again the $z$-axis corresponds to one of the separatrices of $\fol$.
Taking $z = e^{2\pi it}$, it follows that the holonomy map $h$
associated to $\fol$, with respect to the $z$-axis, is given by the time-one map associated to the vector field
\begin{equation}\label{eqholonomy2}
\begin{cases}
\frac{dx}{dt} & = \frac{dx}{dz} \frac{dz}{dt} = -2\pi i x(1 + e^{6\pi it}x^2y)\\
\frac{dy}{dt} & = \frac{dx}{dz} \frac{dz}{dt} = -2\pi i y(1 - e^{6\pi it}x^2y)
\end{cases} \, .
\end{equation}
The same argument employed in Lemma~\ref{preservingxy} shows that, again, the holonomy map $h$ in question preserves
the level sets of the function $(x,y) \mapsto xy$. To solve the corresponding system of equations, we then consider again the series expansion
of $(x(t) ,y(t))$ in terms of the initial condition. Let then $x(t) = \sum a_{ij} (t) x_0^i y_0^j$ and $y(t) = \sum b_{ij} (t) x_0^i y_0^j$, where
$a_{10} (0) = b_{01} (0) =1$ and $a_{ij} (0) =b_{ij} (0) =0$ in the remaining cases. It can immediately be checked that the functions
$a_{ij}, b_{ij}$ vanish identically for $2 \leq i+j \leq 3$. As to the monomials of degree~$4$, it can similar be checked that they all
vanish identically except $a_{31} (t)$ and $b_{22} (t)$. In fact, the latter functions satisfy
\begin{equation}
\begin{cases}
a_{31}^{\prime} (t)& =  -2\pi i [ a_{31} (t) +  e^{6\pi i} a_{10}^3 (t) b_{01} (t)] \\
b_{22}^{\prime} (t)& = -2\pi i [ b_{22} (t) - e^{6\pi i} a_{10}^2 (t) b_{01}^2 (t)]
\end{cases} \, .
\end{equation}
ensuring that $a_{31} (t) = -2\pi i t e^{-2\pi it}$ whereas $b_{22} (t) = 2\pi i t e^{-2\pi it}$. In particular, $a_{31} (1) = -2\pi i$ and
$b_{22} (1) = 2\pi i$. By using induction on $i+j$ (and keeping in mind that $h$ preserves the function $(x,y) \mapsto xy$), it can
be shown that $h$ takes on the form
\begin{equation}\label{eqexpressionH2}
h(x,y) = (x(1 + x^2y f(x^2y)), y(1 + x^2y f(x^2y))^{-1}) \; \; \, {\rm with} \; \; \, f(0) =2\pi i \, .
\end{equation}
It follows from the discussion in Example~2 that this diffeomorphism possesses finite orbits whereas it is clearly non-periodic.
This finishes the proof of the complement to Theorem~B.

\vspace{0.1cm}

\noindent {\bf Example 7}: A dicritical holonomy map.

Let us finish this section with an example of foliation whose holonomy map has a peculiar property related to ``dicritical vector fields''.
As it will follow from our discussion based on a result due to Abate \cite{abate}, the mentioned holonomy map will possess infinite orbits
and, in fact, the basin of attraction of the origin will have non-empty interior. We start with the following vector field
\[
X = x(1 + xyz^2) \frac{\partial }{\partial x} + y(1 + xyz^2) \frac{\partial }{\partial y} - z \frac{\partial }{\partial z} \, .
\]
Compared to Example~5, the reader will note the change of a sign in the second coordinate. The system leading to the holonomy map
associated to the $z$-axis, becomes
\begin{equation}\label{eqholonomy3}
\begin{cases}
\frac{dx}{dt} & = \frac{dx}{dz} \frac{dz}{dt} = -2\pi i x(1 + e^{4\pi it}xy)\\
\frac{dy}{dt} & = \frac{dx}{dz} \frac{dz}{dt} = -2\pi i y(1 + e^{4\pi it}xy)
\end{cases} \, .
\end{equation}
Considering the series expansions
$x(t) = \sum a_{ij} (t) x_0^i y_0^j$ and $y(t) = \sum b_{ij} (t) x_0^i y_0^j$, where
$a_{10} (0) = b_{01} (0)=1$, it can be show that terms $a_{ij}, \, b_{ij}$ with $i+j =2$ vanish identically. As to $i+j=3$, the only
non-identically zero functions are $a_{21}$ and $b_{12}$ which, in turn, satisfy the equations
\begin{equation}
\begin{cases}
a_{21}^{\prime} (t) & =  -2\pi i [ a_{21} (t) +  e^{4\pi i} a_{10}^2 (t) b_{01} (t)] \\
b_{12}^{\prime} (t) & = -2\pi i [ b_{12} (t) + e^{4\pi i} a_{10} (t) b_{01}^2 (t)]
\end{cases} \, .
\end{equation}
yielding $a_{21} (t)= b_{12} (t) = -2\pi i t e^{-2\pi it}$. In particular, $a_{21} (1) = b_{12} (1) = -2\pi i$. Therefore
$$
h (x,y) = (x,y) + xy (-2\pi i x, -2\pi i y) + \cdots \, 
$$
where the dots stand for higher order terms.
The fact that the first non-zero homogeneous component of $h(x,y)-(x,y)$ is a multiple of $(x,y)$ means that $h$ is
{\it dicritical}\, in the terminology of \cite{abate}. This fact is equivalent to saying that the infinitesimal generator
of $h$ is a formal {\it dicritical} vector field, namely its first non-zero homogeneous component is a multiple of the radial
vector field, cf. Section~4. In \cite{abate}, it is proved that these vector fields possess uncountably many ''parabolic domains''
implying, in particular, that the basin of attraction of the origin, with respect to $h$, has non-empty interior.


\section{Solvable and Pseudo-solvable subgroups of $\formdiffn$}

As already mentioned, this section owns a large deal to the paper \cite{ghysBSBM} by E. Ghys and deals with his notion of
``pseudo-solvable groups''. As far as Theorem~C is concerned, the purpose of this section is to show that a ``pseudo-solvable''
subgroup of $\diffCtwo$ is, indeed, solvable. In dimension~$1$ the corresponding result is established in \cite{ghysBSBM}.
The argument presented in this section roughly follows Ghys's strategy in the one-dimensional case although it becomes far more
involved due to the possible existence of non-constant first integrals and to the existence of rank~$2$ abelian groups: these
phenomena have no one-dimensional analogue. In fact, material
involving formal diffeomorphisms, vector fields and solvable groups are widely developed in dimension~$1$, cf. for example \cite{nakai}, \cite{russians}
and \cite{cerveaumoussu}, but not so much in higher dimensions as basic
results such as Lemmas~\ref{commuting1} and~\ref{commuting2} are hardly found in the literature. In the course of the discussion we shall
also provide detailed information on the algebraic structure of solvable subgroups $\diffCtwo$ and the corresponding material may have
interest beyond the use made in this paper.

Let us first recall the definition of pseudo-solvable groups. Suppose that
$G$ is a subgroup of $\diffCtwo$ generated by a finite set $S \subset \diffCtwo$. To the
generating set $S$ it is associated
a sequence of sets $S(j) \subseteq G$ as follows: $S(0) =S$ and $S(j+1)$ is the set whose elements are the
commutators written under the form $[F_1^{\pm 1} ,F_2^{\pm 1}]$ where $F_1 \in S(j)$ and $F_2 \in S(j) \cup S(j-1)$ ($F_2 \in S(0)$ if $j=0$).
The group $G$ is said to be pseudo-solvable if (for some generating set $S$ as above), the sequence $S(j)$
becomes reduced to the identity for $j$ large enough. As mentioned, this section is wholly devoted to proving the following:

\begin{prop}
\label{commuting9}
A pseudo-solvable subgroup $G$ of $\diffCtwo$ is necessarily solvable.
\end{prop}

To begin the approach to Proposition~\ref{commuting9}, let $\formalC$ denote the space of formal series in the variables $x,y$. Similarly
$\fieldC$ will stand for  field of fractions (or field of quotients) of $\formalC$.
Let $\ghatXone$ denote the set of formal vector fields at $(\C^2,0)$. This means that an element (formal vector field) in $\ghatXone$ has the form
$a(x,y) \partial /\partial x + b(x,y) \partial /\partial y$ where $a, \, b \in \formalC$. The space of formal vector fields whose first jet at the origin
vanishes is going to be denoted by $\ghatX$.
Formal vector fields as above act as derivations on $\formalC$ by the formula
$X_{\ast} f = df. X \in \ghatXone$, where $f \in \formalC$ and $X \in \ghatXone$. This action can naturally be iterated so that $(X)^k_{\ast} f$
is inductively defined by $X_{\ast} [ (X)^{k-1}_{\ast} f]$ for $k \in \N$. By means of definition, we also set $(X)^0_{\ast} f =f \in \formalC$.

Next, let $t \in \C$ and $X \in \ghatXone$ be fixed. The {\it exponential of $X$ at time-$t$}, $\exp (tX)$, can be defined as the operator
from $\formalC$ to itself given by
\begin{equation}
\exp (tX) (h) = \sum_{j=0}^{\infty} \frac{t^j}{j!} (X)^j_{\ast} h \, . \label{betterasformula}
\end{equation}
Naturally $\exp (0.X)$ is the identity operator and $\exp (t_1 X) \circ \exp (t_2X) = \exp ((t_1+t_2)X)$.

Recall that the order of a function (or vector field) at the origin is
the degree of its first non-zero homogeneous component. Suppose then
that $X \in \ghatX$ so that $X =a(x,y) \partial /\partial x + b(x,y) \partial /\partial y$ where the orders of both $a, \, b$
at $(0,0) \in \C^2$ are at least~$2$. It then follows that the order of $X_{\ast} h$ is strictly greater
than the order of $h$ itself. In particular, for $h=x$, we conclude that
\begin{equation}
\exp (tX) (x) = x + t.a(x,y) + \cdots \; \; \, {\rm and} \; \; \,  \exp (tX) (y) = y + t.b(x,y) + \cdots \label{previousformula}
\end{equation}
where the dots stand for terms whose degrees in $x,y$ are strictly greater than the order of $a$ (resp. $b$) at the origin.
Therefore, for every $X \in \ghatX$ and every $t \in \C$, the pair of formal
series $(\exp (tX)(x), \exp (tX) (y))$ can be viewed as an element of $\formdiffn$, namely the group of formal diffeomorphisms
of $(\C^2,0)$ that are tangent to the identity at the origin. If the vector field $X$ happens to be holomorphic, as opposed to formal,
then $(\exp (tX)(x), \exp (tX) (y))$
is an actual diffeomorphism tangent to the identity and coinciding with the diffeomorphism induced by the local flow of $X$ at time~$t$.
Next, by letting ${\rm Exp} \, (X) = (\exp (X)(x), \exp (X) (y))$, and more generally, ${\rm Exp} \, (tX) = (\exp (tX)(x), \exp (tX) (y))$,
the following well-known lemma holds:
\begin{lemma}
\label{correspondence}
The map ${\rm Exp}$ settles a bijection between $\ghatX$ and $\formdiffn$.
\end{lemma}

\begin{proof}
In the sequel $p_n (x,y), \, q_n (x,y) , \, a_n (x,y), \, b_n (x,y)$ denote homogeneous polynomials of degree~$n$ in the variables
$x,y$. Let $F \in \formdiffn$ be given by $F(x,y) = (x + \sum_{n=2}^{\infty} p_n (x,y) , y + \sum_{n=2}^{\infty} q_n (x,y))$. Similarly
consider a vector field $X \in \ghatX$ given as
$$
X = \sum_{n=2}^{\infty} \left[ a_n (x,y) \frac{\partial}{\partial x} +  b_n (x,y) \frac{\partial}{\partial y} \right] \, .
$$
The equation ${\rm Exp}\, (X) =F$ amounts to $p_{m+1} = a_{m+1} + R_{m+1} (x,y)$ and
$q_{n+1} = b_{n+1} + S_{m+1} (x,y)$ where $R_{m+1} (x,y)$ (resp. $S_{m+1} (x,y)$) stands for the homogeneous component
of degree~$m+1$ of the vector field
$$
\sum_{j=2}^m \frac{1}{j!} (Z_m)^j (x)
$$
(resp. of $\sum_{j=2}^m (Z_m)^j (y) /  j!$), where $Z_m = \sum_{n=2}^{m} [ a_n (x,y) \partial /\partial x +  b_n (x,y) \partial /\partial y ]$.
These equations show that, given $F \in \formdiffn$, there is one unique $X \in \ghatX$ such that
${\rm Exp} \, (X) =F$. The lemma is proved.
\end{proof}

For $F \in \formdiffn$, recall that the formal vector field $X$ satisfying ${\rm Exp} \, (X) =F$ is called the {\it infinitesimal generator of $F$}.
The notation $X = \log \, (F)$ may also be used to state that $X$ is the infinitesimal generator of $F \in \formdiffn$.
Note that, in general,
the series of $X$ is not convergent even when $F$ is an actual holomorphic diffeomorphism. Now, we have:

\begin{lemma}
\label{commuting1}
Two elements $F_1,\, F_2$ in $\formdiffn$ commute if and only if so do their infinitesimal generators $X_1, \, X_2$.
\end{lemma}

\begin{proof}
The non-immediate implication consists of showing that $[X_1, X_2] =0$ provided that $F_1$ and $F_2$ commute.
For this, denote by $Z_+$ (resp. $Z_-$) the infinitesimal generator of $F_1 \circ F_2$ (resp. $F_1^{-1} \circ F_2^{-1}$). The
diffeomorphisms $F_1, F_2$ commute if and only if $F_1 \circ F_2 \circ F_1^{-1} \circ F_2^{-1} = {\rm Exp} \, (Z_+) {\rm Exp} \, (Z_-)
= {\rm id}$. Denoting by $Z$ the infinitesimal generator of $F_1 \circ F_2 \circ F_1^{-1} \circ F_2^{-1}$, we have
\begin{eqnarray*}
Z & = & \log \, ( {\rm Exp} \, (Z_+) {\rm Exp} \, (Z_-) ) = \\
 & = & Z_+ + Z_- + \frac{1}{2} [Z_+, Z_-] + \frac{1}{12} [Z_+ ,[Z_+,Z_-]] - \frac{1}{12} [Z_- ,[Z_+,Z_-]] + {\rm h.o.t.}
\end{eqnarray*}
as it follows from Campbell-Hausdorff formula, see \cite{who??}. In turn,
$$
Z_+  =  \log \, ( F_1 \circ F_2 ) = \log \, ( {\rm Exp} \, (X_1) {\rm Exp} \, (X_2)) =
 X_1 + X_2 + \frac{1}{2} [X_1, X_2] + \cdots \, .
$$
Analogously
$$
Z_- = -X_1 -X_2 + \frac{1}{2} [X_1, X_2] + \cdots \, .
$$
Therefore
\begin{eqnarray}
Z & = & X_1 + X_2 + \frac{1}{2} [X_1, X_2] + \cdots  + ( -X_1 -X_2 + \frac{1}{2} [X_1, X_2] + \cdots ) +  \nonumber \\
 & & + \frac{1}{2} \left[X_1 + X_2 + \frac{1}{2} [X_1, X_2] + \cdots , -X_1 -X_2 + \frac{1}{2} [X_1, X_2] + \cdots \right] + \cdots \nonumber \\
 & = & [X_1,X_2] + \frac{1}{8} [[X_1,X_2], [X_2,X_1]] + \cdots \, . \label{finalcommutators}
\end{eqnarray}
Hence, if $X_1,X_2$ do not commute, then $[X_1, X_2]  = \sum_{j \geq k}^{\infty} Y_j$,
where $Y_j$ are homogeneous vector field and $k$ is the smallest strictly positive integer for which $Y_k$ is not identically {\it zero}. The orders
of the higher iterated commutators appearing in Equation~(\ref{finalcommutators}) are strictly greater than $k$, since
the orders of $X_1,X_2$ at the origin are at least~$2$. In other words, we have $Z = Y_k + {\rm h.o.t.}$. Since
$F_1 \circ F_2 \circ F_1^{-1} \circ F_2^{-1} = {\rm Exp} \, (Z)$, it follows that $F_1, F_2$ do not commute. The lemma is proved.
\end{proof}

\begin{obs}
\label{obs1.1}
{\rm Consider elements $F_1, F_2 \in \formdiffn$ whose orders of contact with the identity at $(0,0)$ are respectively $r, s$.
The preceding argument shows, in particular, that the order of contact with the identity at $(0,0)$ of $F_1\circ F_2 \circ F_1^{-1} \circ F_2^{-1}$
is at least $r+s-1$. Since $F_1,F_2$ are tangent to the identity, we have that $\min \{ r,s\} \geq 2$ so that the order of contact with the identity
of $F_1\circ F_2 \circ F_1^{-1} \circ F_2^{-1}$ is strictly greater than $\max \{ r,s \}$}.
\end{obs}

Armed with Lemma~\ref{commuting1}, it is now easy to describe the set consisting of those elements in $\formdiffn$
commuting with $F$, namely the centralizer of an element $F \in \formdiffn$. First, given
formal vector fields $X, Y$, we shall say that they are {\it not everywhere parallel}\, to mean that $X$ is not a multiple of $Y$ by an element in $\fieldC$.
Also, given a formal vector field $X$, there may or may not exist another vector field $Y$ not everywhere parallel to $X$ and commuting with $X$.
When this vector field $Y$ exists, it is never unique since every linear combination of $X$ and $Y$ has similar properties. Furthermore, if $X$ happens
to admit some non-constant first integral $h$, then $hY$ will also commute with $X$. Although $Y$ and $hY$ belong to $\ghatX$, the possibility of
having $h$ in $\fieldC \setminus \formalC$ cannot be ruled out since $Y$ need not have isolated singularities.

Also, preparing the way for Lemma~\ref{commuting2} below, note that every element of $\formdiffn$ possessing an infinitesimal generator $Z$
of the form $Z=aX + bY$, where $X, Y$ are as above and $a,\, b$ are first integrals of $X$, automatically belongs to the centralizer of $X$, cf. Lemma~\ref{commuting1}.
Clearly the vector field $Y$ (commuting with $X$ and not everywhere parallel to $X$) is never uniquely defined but a representative of them can be chosen.
Once this vector field $Y$ is chosen, we also have the following:

\begin{lemma}
\label{lastversionLemma1}
Let $X, Y$ be as above and consider the set of elements $F \in \formdiffn$ whose infinitesimal generators have the form
 $Z=aX + bY$, where $a,\, b$ are first integrals of $X$. Then the set of these elements $F \in \formdiffn$ form a subgroup of
 $\formdiffn$.
\end{lemma}

\begin{proof}
In what follows it is understood that $a,b,c,d$ are always first integrals of $X$. To show that
the set of elements in $\formdiffn$ having infinitesimal generators of the form $aX +bY$ is a group, consider vector fields $Z_1 = aX +bY$ and $Z_2 = cX + dY$ and set
$F_1 = {\rm Exp}\, (Z_1)$, $F_2 = {\rm Exp}\, (Z_2)$. According to Campbell-Hausdorff formula, the infinitesimal generator $Z$ of $F_1 \circ F_2$
is given by
$$
Z = Z_1  + Z_2 + \frac{1}{2} [Z_1,Z_2] + \frac{1}{12} \left( [Z_1, [Z_1,Z_2]] - [Z_2, [Z_1,Z_2]] \right) + \cdots \, .
$$
However, note that
$$
[Z_1,Z_2] = \left( \frac{\partial a}{\partial Y} - \frac{\partial c}{\partial Y} \right) X + \left( \frac{\partial b}{\partial Y} - \frac{\partial d}{\partial Y} \right) Y \, .
$$
Since $X, Y$ commute, Schwarz theorem implies that the derivative with respect to $Y$ of a first integral for $X$ still is a first integral for $X$.
Thus $[Z_1,Z_2]$ has the form $\tilde{a} X + \tilde{b} Y$, where $\tilde{a}, \tilde{b}$ are first integrals for $X$. Now an induction argument immediately
ensures that the same conclusion is valid for the higher iterated commutators appearing in Campbell-Hausdorff formula. Therefore the
infinitesimal generator of $F_1 \circ F_2$ still has the general form $AX + BY$ where $A,B$ are first integrals of $X$. The lemma is proved.
\end{proof}

The next lemma characterizes the centralizer of $X$ and shows, in particular, that the group described in Lemma~\ref{lastversionLemma1} does not
depend on the choice of the representative vector field $Y$.

\begin{lemma}
\label{commuting2}
Let $F \in \formdiffn$ be given and denote by $X$ its infinitesimal generator. Then the centralizer of $F$ in $\formdiffn$
coincides with one of the following groups.
\begin{description}
\item[{\sc Case 1}] Suppose that every vector field $Y \in \ghatX$ commuting with $X$ is everywhere parallel to $X$. Then the centralizer
of $X$ consists of the subgroup of $\formdiffn$ whose elements have infinitesimal generators of the form $hX$, where $h \in \fieldC$
is a formal first integral of $X$. In particular, if $X$ admits only constants as first integrals, then the centralizer of $F$ is reduced
to the exponential of $X$.

\item[{\sc Case 2}] Suppose there is $Y \in \ghatX$ which is not everywhere parallel to $X$ and still commutes with $X$.
Then the centralizer of $F$ coincides with the subgroup of $\formdiffn$ consisting of those elements $F \in \formdiffn$ whose
infinitesimal generators have the form $aX +bY$, where $a,b \in \fieldC$ are (formal) first integral of $X$.
\end{description}
\end{lemma}

\begin{proof}
Suppose that $H$ is an element of $\formdiffn$ commuting with $F$. Denoting by $Z$ the infinitesimal generator of $H$, it follows
from Lemma~\ref{commuting1} that $[X,Z]=0$. Conversely the $1$-parameter group generated by $Y$ is automatically contained in
the centralizer of $F$.

Next, suppose that the assumption in Case~1 is verified. Then the quotient $h$ between $Z$ and $X$ can be defined as an element
of $\fieldC$ satisfying $Z = hX$. Therefore the condition $[X,Z]=0$ becomes $dh.X=0$, i.e. $h$ is a first integral for $X$.

Consider now the existence of $Y$, not everywhere parallel to $X$, verifying $[X,Y]=0$.  It is clear that the elements of $\formdiffn$ described
in Case~2 belong to the centralizer of $\fol$. Thus only the converse needs to be proved. Since $H$ commutes with $F$, Lemma~\ref{commuting1}
yields again $[X,Z]=0$.
Since $Y$ is not a multiple of $X$, there are functions $a(x,y) , \, b(x,y) \in \fieldC$ such that $Z = a X + b Y$. Now the equation $[X, Z] = 0$
yields
$$
( \partial a /\partial X) .X +  ( \partial b /\partial X) .Y =0 \, .
$$
Thus the fact that $Y$ is not a multiple of $X$ ensures that $( \partial a /\partial X)  = ( \partial b /\partial X) =0$. In other words, both
$a, \, b$ are first integrals of $X$. The lemma follows.
\end{proof}

Concerning the situation described in Case~2 of Lemma~\ref{commuting2}, it is already known that $Y$ is not uniquely defined.
Nonetheless the characterization of the subgroup of $\formdiffn$ mentioned in Lemma~\ref{lastversionLemma1} as the centralizer of
$F ={\rm Exp}\, (X)$ implies that the group in question does not depend on the choice of the vector field $Y$ commuting with $X$ and
not everywhere parallel to $X$.

A consequence of Lemma~\ref{commuting2} is as follows.

\begin{lemma}
\label{commuting3}
An abelian group $G \subset \formdiffn$ is either contained in the group generated by the exponentials of two commuting vector fields $X, Y$ that are not everywhere parallel
or it is contained in the group constituted by the exponentials of vector fields $h.X$ (where, as mentioned, $h$ is a first integral of $X$).
\end{lemma}

\begin{proof}
If all elements in $G$ have infinitesimal generators of the form $hX$, where $h$ is a first integral for $X$, then the statement is clear. Thus
suppose there are elements $F_1, F_2$ in $G$ whose respective infinitesimal generators $X, Y$ are not everywhere parallel.
Let $\varphi$ be an element of $G$ whose infinitesimal generator $Z$ is not everywhere parallel to $X$. Clearly we have
$Z = f_1(x,y) X + f_2(x,y) Y$ where $f_1,f_2 \in \fieldC$. Since $\varphi$ must commute with $F_1$, it follows that $[X,Z]=0$ what,
in turn, implies that both $f_1, \, f_2$ are first integrals for $X$ (cf. proof of Lemma~\ref{commuting2}). Analogously $\varphi$ also commutes
with $F_2$ so that $[Y,Z]=0$ and hence $f_1, f_2$ are first integrals for $Y$ as well. Since $X, \, Y$ are not everywhere parallel, it follows
that both $f_1, \, f_2$ must be constant. The lemma is proved.
\end{proof}

\begin{obs}
\label{lastversionRem1}
{\rm Given two commuting non-everywhere parallel vector fields $X, Y$, the subgroup of $\formdiffn$ generated by the exponentials
${\rm Exp}\, (tX), \, {\rm Exp}\, (tY)$, $t\in \C$, of $X, Y$ is going to be referred to as the {\it linear span of $X, \, Y$}.
A consequence of the preceding proof is that every vector field $Z$ commuting with both $X, Y$ must be contained in the linear span of $X,Y$
provided that $X,Y$ are non-everywhere parallel commuting vector fields as above.
In particular, if $F$ is an element of $\formdiffn$ commuting with both
${\rm Exp}\, (X), \, {\rm Exp}\, (Y)$ then the infinitesimal generator $Z$ of $F$ has the form $c_1X + c_2Y$ where $c_1,c_2$ are {\it constants}.}
\end{obs}

Here is another easy consequence of Lemmas~\ref{commuting2}.
\begin{lemma}
\label{lastversionLemma2}
Suppose that $h$ is a non-constant first integral of $X$ and
let $F_1 = {\rm Exp} \, (X)$ and $F_2 = {\rm Exp} \, (hX)$ be elements in $\formdiffn$ (in particular the first jet of $X$ at the origin vanishes).
The intersection of the centralizers of $F_1$ and $F_2$, i.e. the set of elements in $\formdiffn$ commuting with both $F_1, F_2$ is the subgroup of
$\formdiffn$ constituted by those elements whose infinitesimal generators have the form $aX$, where $a$ is a first integral of $X$. In particular,
this group is abelian.
\end{lemma}

\begin{proof}
If every vector field $Y$ commuting with $X$ is everywhere parallel to $X$, then the statement follows at once from Lemma~\ref{commuting2},
Case~1. Suppose then the existence of $Y$ not everywhere parallel to $X$ satisfying $[X,Y]=0$. Again it follows from Lemma~\ref{commuting2} that
the centralizer of $F_1 = {\rm Exp} \, (X)$ consists of those elements whose infinitesimal generators
have the form $aX + bY$. Nonetheless, the elements in the intersection of the centralizers of $F_1, \, F_2$
must commute with $F_2$ as well. According to Lemma~\ref{commuting1}, this happens if and only
if the infinitesimal generator $aX + bY$ commutes with $hX$. However
$$
[hX, aX + bY] = b\left( \frac{\partial h}{\partial Y} \right) X = 0 \, .
$$
Nonetheless $\partial h /\partial Y$ is not identically zero since $Y$ is not everywhere parallel to $X$. It then follows that $b$ must vanish identically. The lemma is proved.
\end{proof}

Another central ingredient in the proof of Proposition~\ref{commuting9} is the algebraic description of solvable subgroups of $\formdiffn$ that
will be undertaken below. As mentioned the corresponding material is interesting in itself.

Recall that the {\it normalizer}\, of a group
$G \subset \formdiffn$ is the maximal subgroup $N_G$ of $\formdiffn$ containing $G$ and such that $G$ is a normal subgroup of $N_G$.
Similarly the {\it centralizer} of an abelian group $G$ is the maximal abelian subgroup of $\formdiffn$ containing $G$ (in particular the
centralizer of an element is nothing but the centralizer of the cyclic group generated by this element).

\begin{lemma}
\label{lastversionLemma3}
Let $G \subset \formdiffn$ be an abelian group contained in the linear span of two non-everywhere parallel commuting vector fields
$X, Y$ (by assumption this also means that $G$ is not contained in the exponential of a single vector field). Then the normalizer $N_G$ of $G$
coincides with the group of diffeomorphisms induced by the linear span of $X, Y$. In particular, $N_G$ is abelian.
\end{lemma}

\begin{proof}
Consider an element $F$ in $N_G$ along with its adjoint action on vector fields belonging to
the linear space $E$ spanned by $X, Y$ which is naturally isomorphic to $\C^2$. This action is clearly well-defined since $G$ is not contained
in the exponential of a single vector field. Next, note that the eigenvalues
of the automorphism of $E$ induced by $F$ are equal to~$1$ since $F$ is tangent to the identity. Then either the action induced by $F$ is the
identity or it is non-diagonalizable. In the former case, $F$ clearly belongs to the linear span of $X\, Y$ and there is nothing else to be proved.
Thus only the non-diagonalizable case remains to be considered. Modulo a change of basis, we can assume that $F^{\ast} X =X$.
Therefore $F$ preserves $X$ and, hence, it is contained in the centralizer of ${\rm Exp}\, (tX)$. According
to Lemma~\ref{commuting2}, the centralizer of ${\rm Exp}\, (tX)$ consists of elements having the form ${\rm Exp}\, (thX)$ or the form
${\rm Exp}\, (thZ)$, where $Z$ is some vector field commuting with $X$ but not everywhere parallel to $X$. In both cases $h$ is a first
integral for $X$. To obtain a contradiction with the assumption that the action of $F$ on $E$ is not diagonalizable,
let us first consider the action of an element $F$ of $\formdiffn$ having the form
$F={\rm Exp}\, (hZ)$ (i.e. $F={\rm Exp}\, (thZ)$ with $t=1$) on the vector field $Z$ itself. According to Hadarmard's lemma, see \cite{who??}, we have
$$
F^{\ast} Z = Z + [hZ, Z] + \frac{1}{2!} [hZ, [hZ,Z]] + \frac{1}{3!} [hZ,[hZ, [hZ,Z]]] + \cdots \, .
$$
By assumption $hZ$ lies in $\ghatX$ so that $F^{\ast} Z-Z$ is a vector field in $\ghatX$ as well. Besides, this vector field is clearly parallel to $Z$.
Thus we have $F^{\ast} Z = Z + Z'$ where $Z' \in \ghatX$ is everywhere parallel to $Z$.
However  $F^{\ast} Z$ is still contained in the vector space $E$ and, furthermore, the action of $F$ on $E$ has eigenvalues equal to~$1$.
Therefore we must have $F^{\ast} Z =Z$ so that $F$ preserves $Z$ contradicting the fact that the action of $F$ on $E$ is not the identity.
Hence we are led to the conclusion that $F$ must have the form ${\rm Exp}\, (thX)$ in order to have a non-diagonalizable action.
Still considering the action of $F = {\rm Exp}\, (hX)$ (i.e. $t$ was again
set to be~$1$) on $Y$, the same Hadamard's lemma yields
$$
F^{\ast} Y = Y + [hX, Y] + \frac{1}{2!} [hX, [hX,Y]] + \frac{1}{3!} [hX,[hX,[hX,Y]]] + \cdots \, .
$$
However $[hX, Y] = (\partial h/\partial Y) X$ and $\partial h/\partial Y$ still is a first integral of $X$ thanks to Schwarz theorem ($X,Y$ commute). In particular
$[hX, [hX,Y]]$ vanishes identically and so do all higher orders commutators. It then follows that
$F^{\ast} Y = Y + ( \partial h/\partial Y)  X$. Besides $\partial h/\partial Y$ must be a constant in $\C$ since $F$ preserves $E$.
By the standard Jordan form, the constant in question can be chosen equal to~$1$ (recall that the action of $F$ on $E$ is not diagonalizable).
Summarizing the preceding discussion, we have $X$ and $Y$ so that
\begin{equation}
F^{\ast} X = X \; \; \, {\rm and} \; \; \, F^{\ast} Y = Y + ( \partial h/\partial Y)  X = Y +X \, . \label{thetwoequations}
\end{equation}
The proof of the lemma is now reduced to showing that every $F$ satisfying the preceding equations must belong to the linear span of $X, \, Y$.
For this suppose that $h$ is holomorphic so that the space of its level curves can be considered. The equation $\partial h/\partial Y =1$ implies that
$Y$ can be projected in the space of level curves of $h$ and, in fact, under this projection $Y$ is mapped to the constant vector field $Z$ (equal to~$1$).
Thus $Y$ decomposes as a multiple of $X$ plus a ``constant transverse'' vector field represented by $Z$. Since $F$ preserves $X$, it must also define an automorphism
on the space of level curves of $h$. Furthermore the condition $F^{\ast} Y = Y +X$ ensures that this induced automorphism must preserve $Z$. Since this
``leaf space'' has dimension~$1$, we conclude that the mentioned induced automorphism must be
embedded in the flow of $Z$. Hence $F$ can itself be decomposed as an automorphism preserving each level curve of $h$ composed with a transverse
automorphism embedded in the flow of $Z$. Since $Y$ commutes with $X$, the leaf-preserving component of $F$ must preserve the vector field $X$
(recall that $F^{\ast} X = X$) and hence it is embedded in the flow of $X$. Summarizing, we concluded that $F$ is obtained by composing a local diffeomorphism
embedded in the flow of $X$ with another local diffeomorphism embedded in the flow of $Y$. The statement is then proved.

The only point in the above discussion where $h$ was required to be analytic was to make sense of its level curves. The general algebraic statement however
also holds in the formal category. For example, by truncating $h$ at some high order, the equation~$F^{\ast} Y = Y + ( \partial h/\partial Y)  X = Y +X$ will still hold
for terms of lower order. Thus every element $F$ in $\formdiffn$ verifying Equation~(\ref{thetwoequations}) must coincide with elements in the linear
span of $X, Y$ to arbitrarily large orders. From this it is straightforward to conclude that $F$ must be contained in the linear span in question.
The proof of the lemma is over.
\end{proof}

The next lemma completes the description of the normalizers of non-trivial abelian groups, cf. Lemma~\ref{commuting3}.

\begin{lemma}
\label{commuting7nowlemma}
Let $G \subset \formdiffn$ be a finitely generated non-trivial abelian group all of whose elements have infinitesimal generators parallel
to a certain formal vector field $X$. Then the normalizer $N_G$ of $G$ in $\formdiffn$ satisfies the following:
\begin{itemize}

\item Suppose that $G$ is contained in ${\rm Exp}\, (tX)$. Then $N_G$ coincides with the centralizer of its elements. Namely, it consists
of those elements of $\formdiffn$ whose infinitesimal generators have the form $aX +  bY$, with $a,b$ first integrals of $X$ and where $Y$
is a vector field commuting with $X$ and not everywhere parallel to $X$ (if $Y$ does not exist, then $N_G$ is reduced to the group formed
by those elements whose infinitesimal generators have the form $aX$).

\item If $G \subset \formdiffn$ is not contained in the exponential of a single vector field $X$, then $N_G$ coincides with the subgroup
of $\formdiffn$ consisting of those elements whose infinitesimal generators have the form $hX$ ($h$ first integral for $X$). In particular,
$N_G$ is abelian.
\end{itemize}
\end{lemma}

\begin{proof}
Suppose first that $G$ is contained in ${\rm Exp} \, (tX)$, for some $X \in \ghatX$ and let $\varphi$ be a non-trivial element of $G$.
Next, let $F \in \formdiffn$ satisfy $F \circ \varphi \circ F^{-1}
= \widetilde{\varphi} \in G$. Then $F \circ \varphi \circ F^{-1} \circ \varphi^{-1} = (\widetilde{\varphi} \circ \varphi^{-1}) \in G$.
Suppose now that $\varphi \neq \widetilde{\varphi}$.
Since $\widetilde{\varphi}, \, \varphi$ are both embedded in ${\rm Exp} \, (tX)$ and satisfy $\widetilde{\varphi} \circ \varphi^{-1} \neq {\rm id}$,
the order of contact between $\widetilde{\varphi} \circ \varphi^{-1}$ and the identity
is equal to the order of contact between $\varphi$ and the identity. On the other hand, the order of contact of $F \circ \varphi \circ F^{-1} \circ \varphi^{-1}
\neq {\rm id}$ with the identity must be strictly larger than the corresponding order of $\varphi$, cf. Remark~\ref{obs1.1}.
The resulting contradiction ensures that $F$ must commute with $\varphi$. Therefore $N_G$ coincides with the centralizer
of $\varphi$ and the statement results from Lemma~\ref{commuting2}.

Suppose now that $G$ is contained in the group generated by a number of exponentials ${\rm Exp}\, (thX)$, where $h$ is a first integral for $X$,
but not on the exponential of a single vector field.
Since $G$ is finitely generated and abelian, it contains an element $\varphi \neq {\rm id}$ having maximal
order $r$ of contact with the identity. Without loss of generality, let $X$ denote the infinitesimal generator of $\varphi$.
Note that $G$ may contain other elements having order of contact
with the identity equal to~$r$ but not contained in ${\rm Exp}\, (tX)$ since $X$ might admit a first integral in $\fieldC$ whose order of the ``numerator'' equals the
order of the ``denominator''. Yet, again the fact that $G$ is finitely generated implies that only finitely many vector fields $X, h_1 X,  \ldots , h_l X$ may give rise to
elements in $G$ having order of contact with the identity equal to~$r$.
Now, recall that inner automorphisms of $\formdiffn$
preserves the order of contact with the identity so that the normalizer of $G$ must left invariant the set of elements of $G$ contained
in one of the following exponentials: ${\rm Exp}\, (tX), \ldots , {\rm Exp}\, (th_lX)$.
Therefore the natural action of $F \in N_G$ on $\ghatX$ must induce a permutation of the set $X, h_1 X,  \ldots h_l X$. Modulo passing to a finite power $F^k$ of $F$,
it follows that $F^k$ preserves ${\rm Exp}\, (tX)$ and hence it must commute with $\varphi$. In other words, there is a finite power $F^k$ of $F$ that lies in the centralizer
of $\varphi$. However, it follows from the description of centralizers
presented in Lemma~\ref{commuting2} that the ``$k^{\rm th}$-root'' of an element in the centralizer still belongs to the centralizer.
Therefore $F$ itself belongs to the centralizer of $\varphi$. However, if $l \geq 1$, then the power $F^k$ may be chosen so that $F^k$ preserves also
the vector field $h_1X$. Thus the same argument will imply that $F$ belongs to the centralizer of some element in ${\rm Exp}\, (th_1X)$. According
to Lemma~\ref{lastversionLemma2}, the intersection of these two centralizers consist of elements in $\formdiffn$ whose infinitesimal generators
have the form $hX$ where $h$ is a first integral for $X$. The lemma is proved in this case.

Finally suppose that $l=0$ and let $s$ be the ``second greatest'' order of an infinitesimal element $\tilde{h}_0X$ of an element in $G$. In  particular
$2 \leq s < r$. By assumption, this element exists since $l=0$ and $G$ is not contained in the exponential  of a single vector field.
Again, $G$ being finitely generated and abelian, it follows that there are only finitely many formal vector fields
$\tilde{h}_0 X, \ldots ,\tilde{h}_mX$ having order~$s$ at the origin and corresponding to infinitesimal generators of elements in $G$. Therefore
the action of $F$ on the set formed by these vector fields must preserve the set itself. Thus $F$ will finally belong to the centralizer of some element
whose infinitesimal generator is $\tilde{h}_0X$. Lemma~\ref{lastversionLemma2} will then complete our proof.
\end{proof}

\begin{obs}
\label{ageneraldecomposition}
{\rm Assume we are given vector fields $X, Y \in \ghatX$ which are not everywhere parallel. Consider also a third vector field $Z \in \ghatX$.
Since $X,Y$ are not everywhere parallel, there exists a unique decomposition $Z = aX +bY$ with $a, b$ belonging to $\fieldC$. Our purpose
here is to remind the reader of the elementary fact that, whereas $a,b$ may lie in $\fieldC \setminus \formalC$, the corresponding
vector fields $aX, \, bY$ have coefficients in $\formalC$ and, in fact, the vector fields $aX, \, bY$ belong to $\ghatX$,
as it can immediately be checked by directly computing the corresponding functions $a,b$. The contents of this remark will implicitly be used
in the rest of the section.}
\end{obs}

Let us now state a technical lemma that will repeatedly be used in the proofs of the two main results of this
section, namely Propositions~\ref{commuting8} and~\ref{commuting9}
below. For this suppose that $X,Y$ are formal vector fields in $\ghatX$ that are not everywhere parallel.

\begin{lemma}
\label{formerclaim2}
Suppose that $X, Y$ are non-everywhere parallel commuting vector fields.
Consider elements $F_1, F_2 \in \formdiffn$ whose infinitesimal generators are respectively given by
$a_1 X + b_1 Y$ and by $a_2 X + b_2 Y$. The coefficients $a_i,b_i \in \formalC$ are supposed to be first integrals for $X$ and $b_1 . b_2$
is supposed not to vanish identically. Then the infinitesimal generator of $F_1 \circ F_2 \circ F_1^{-1} \circ F_2^{-1}$ is everywhere parallel
to $X$ if and only if the quotient $b_1/b_2$ is constant.
\end{lemma}

\begin{proof}
To begin with, note that the commutator of $a_1X + b_1 Y$ and $a_2 X + b_2 Y$ is given by
$$
[a_1X + b_1 Y, a_2X + b_2 Y] = \left( b_1 \frac{\partial a_2}{\partial Y} - b_2 \frac{\partial a_1}{\partial Y}\right) X +
\left( b_1 \frac{\partial b_2}{\partial Y} - b_2 \frac{\partial b_1}{\partial Y}\right) Y \, .
$$
Assume that the infinitesimal generator of $F_1 \circ F_2 \circ F_1^{-1} \circ F_2^{-1}$ is everywhere parallel
to $X$. To conclude that the quotient $b_1/b_2$ is a constant, it suffices to check that the coefficient
of $Y$ in the right-hand side of the above equation vanishes identically. Indeed, this would mean that
$b_1/b_2$ is a first integral for $Y$ and hence must be constant since it is
also a first integral for $X$ (and $X,Y$ are not everywhere parallel by assumption).
Therefore, it remains to check that the function $b_1 (\partial b_2/\partial Y) - b_2 (\partial b_1/\partial Y)$ is identically zero.
For this let us recall that, whether or not the functions $a_i,b_i$ belong to $\fieldC \setminus \formalC$, $i=1,2$, the vector fields $a_1X + b_1 Y,
\, a_2X + b_2 Y$ are supposed to belong to $\ghatX$. Therefore their commutators have order strictly higher than the maximum between the orders
of $a_1X + b_1 Y$ and of $a_2X + b_2 Y$. Having recalled this fact, the argument to prove that $b_1 (\partial b_2/\partial Y) - b_2 (\partial b_1/\partial Y)$ must
vanish identically is as follows. Suppose for a contradiction that $b_1 (\partial b_2/\partial Y) - b_2 (\partial b_1/\partial Y)$ does not
vanish identically and denote by
$C_m (x,y)$ is first non-zero homogeneous component
(so that $m$ is the order of the function $b_1 (\partial b_2/\partial Y) - b_2 (\partial b_1/\partial Y)$). Whereas $m$ may be negative,
the vector field $[b_1 (\partial b_2/\partial Y) - b_2 (\partial b_1/\partial Y)] Y$ lies in $\ghatX$.

On the other hand, as done in the proof of Lemma~\ref{commuting1}, the Campbell-Hausdorff formula may be used to compute the
infinitesimal generator of $F_1 \circ F_2 \circ F_1^{-1} \circ F_2^{-1}$. It turns out, however, that iterated commutators
higher than $[a_1X + b_1 Y, a_2X + b_2 Y] $ give rise to monomials of degree strictly greater than~$m$ appearing as multiplicative factors
of the vector field $Y$. In fact, all contributions parallel to $Y$ arise from a commutator of the form $[\overline{b}Y,
\overline{d}Y]$ where at least one between $\overline{b}Y, \, \overline{d}Y$ has order greater than or equal to the order of $C_m (x,y).Y$.
It then follows that the component $C_m (x,y).Y$ appearing in the expression for the mentioned infinitesimal generator of 
$F_1 \circ F_2 \circ F_1^{-1} \circ F_2^{-1}$ will not be cancelled out by a contribution arising from higher order commutators.
In other words, the infinitesimal generator of $F_1 \circ F_2 \circ F_1^{-1} \circ F_2^{-1}$ has a non-zero component in $Y$ and thus is not
everywhere parallel to $X$.

The converse is a more direct application of the Campbell-Hausdorff formula. If $b_2 =cb_1$ for some constant $c \in \C$, then 
$[a_1X + b_1 Y, a_2X + b_2 Y] $ is everywhere parallel to~$X$. It is the immediate to check that all higher iterated commutators are everywhere
parallel to~$X$ as well. The proof of the lemma is over.
\end{proof}

Building on the previous material,
the next proposition provides some key information on the structure of solvable subgroups of $\formdiffn$.

\begin{prop}
\label{commuting8}
Suppose that $G \subset \formdiffn$ is a solvable non-abelian group. Then the following holds:
\begin{itemize}
\item $G$ has non-trivial center $Z(G)$.
\item $G$ is metabelian, i.e. its first derived group is abelian.
\item If $G$ is not abelian, then $Z (G)$ coincides with the first derived group $D^1 (G)$ of $G$. Moreover $Z (G) =D^1 (G)$ is fully contained
in ${\rm Exp}\, (tX)$ for a certain $X \in \ghatX$.
\end{itemize}
\end{prop}

\begin{proof}
Consider the derived series $D^0 G = G$, $D^1G = \langle [ G,G] \rangle$, $D^{i+1} G = \langle [D ^i G, D^i G] \rangle$
of $G$. Let $k \geq 1$ be the largest
integer for which $D^{i} G$ is not trivial. Then $D^k G$ is a non-trivial abelian group which, in addition,
is a normal subgroup of $D^{k-1} G$. Our first purpose is to characterize $D^k G$.

\vspace{0.1cm}

\noindent {\it Claim~1}: $D^k G$ is contained in the exponential ${\rm Exp}\, (tX)$ of a single vector field $X$.

\noindent {\it Proof of Claim~1}. Suppose for a contradiction that the statement is false. Then, being an abelian group, it follows
from Lemma~\ref{commuting3} that either $D^k G$ is contained in the group induced by the linear span of two commuting non-everywhere parallel
vector fields~$X,Y$ or it is as in the second item of Lemma~\ref{commuting7nowlemma}. In both cases, it follows from Lemma~\ref{lastversionLemma3}
and Lemma~\ref{commuting7nowlemma}, that the normalizer $N_{D^k G}$ of $D^k G$ is an abelian group. Since $D^{k-1} G \subset N_{D^k G}$,
we conclude that $D^{k-1} G$ is itself abelian what is impossible since $D^k G$ is not reduced to the identity. The claim is proved.\qed

Since $D^{k-1} G$ is not abelian, we conclude in particular the existence of a vector field $Y$ commuting with $X$ and not everywhere parallel
to $X$. Besides, $D^{k-1} G$ is contained in the centralizer of ${\rm Exp}\, (tX)$ so that, every element in $D^{k-1} G$ has an infinitesimal generator
of the form $aX + bY$, where $a,b$ are first integrals of $X$, cf. Lemma~\ref{commuting7nowlemma}.
Consider again the collection of all infinitesimal generators $a_i X + b_iY$,
$i=1, \ldots ,l$ of non-trivial elements in $D^{k-1} G$.

\vspace{0.1cm}

\noindent {\it Claim~2}: There exists one value of $i$ for which $b_i$ is not identically zero.
Furthermore, if $a_{i_1} X + b_{i_1}Y$
and $a_{i_2} X + b_{i_2}Y$ are such that $b_{i_1} b_{i_2}$ is not identically zero, then the
quotient $b_{i_1} /b_{i_2}$ is a constant.

\noindent {\it Proof of Claim~2}. Clearly there is at least one value of $i$ for which $b_i$ is not identically zero, otherwise $D^{k-1} G$
would be an abelian group.
Next consider, without loss of generality, that $a_1 X + b_1Y$ and $a_2 X + b_2Y$ are such that $b_1 b_2$
does not vanish identically. The commutator subgroup of $D^{k-1} G$ being $D^kG$, all its elements have a (constant)
multiple $X$ as infinitesimal generator. The fact that the quotient $b_{i_1} /b_{i_2}$ must be constant then results at once from
Lemma~\ref{formerclaim2}.\qed

Let then $\overline{f}$ denote a non-identically zero function such that, for every infinitesimal
generator $a_i X +b_i Y$ of an element in $D^{k-1} G$, the coefficient $b_i$ is a constant multiple
(possibly zero) of $\overline{f}$. In the sequel, we are going to
show that $G$ is metabelian, i.e. that $k =1$ (provided that $G$ is not abelian).
Naturally this will complete the proof of our proposition. Indeed, it was just seen that $D^{k-1} G$ is contained in the centralizer of ${\rm Exp}\, (tX)$
and, in turn, $D^k G$ is contained in ${\rm Exp}\, (tX)$ (for a unique vector field $X$). The statement is then established provided that $k=1$.

To prove that $G$ is metabelian, let us suppose for a contradiction that $k \geq 2$. Hence
the group $D^{k-2} G$ can be considered. This group contains $D^{k-1} G$ as a normal
subgroup. Since $D^{k-2} G$ normalizes $D^{k-1} G$, it must also normalize the center
of $D^{k-1} G$, namely the group $D^k G$. Therefore $D^{k-2} G$ is contained in the centralizer
of ${\rm Exp}\, (tX)$ and this ensures that the infinitesimal generator of every element in
$D^{k-2} G$ still has the form $cX +dY$, with $c,d$ being first integrals of $X$. The next
step consists of characterizing these elements so as to show that the commutator
between every two elements in $D^{k-2} G$ possesses an infinitesimal generator that is a multiple of $X$. A contradiction with the fact that $k \geq 2$
then arises since there are elements in $D^{k-1} G$ whose infinitesimal generators have the form
$a_i X + b_i Y$ with $b_i$ non identically zero (Claim~2).

Let then $\psi  \in D^{k-2} G$ be a non-trivial element whose infinitesimal generator
is $cX +dY$ and consider another non-trivial element $\varphi \in D^{k-1} G$ whose
infinitesimal generator $a X + bY$ is such that $b$ does not vanish identically.
According to Hadamard lemma, the infinitesimal
generator $AX + BY$ of $\psi \circ \varphi \circ \psi^{-1} \in D^{k-1} G$ is given by
$aX +bY + [cX+dY, aX +bY] + \cdots$
where the dots stand for terms whose order is strictly greater than the order of
$[cX+dY, aX +bY]$. In turn,
\begin{equation}
[cX+dY, aX +bY] =
\left( d \frac{\partial a}{\partial Y} - b \frac{\partial c}{\partial Y}\right) X +
\left( d \frac{\partial b}{\partial Y} - b \frac{\partial d}{\partial Y}\right) Y \, .
\label{therewego}
\end{equation}
Note that the coefficient $B$ in the infinitesimal generator $AX + BY$ is a constant
multiple of $b$ thanks to Claim~2. It then follows that
the coefficient of $Y$ in the right-hand side of Formula~(\ref{therewego}) must vanish identically.
In fact, if it does not vanish identically, its order is strictly greater than the
order of $bY$ since this coefficient is nothing but $[dY, bY]$. On the other hand, the coefficients
of $Y$ in the remaining terms of Hadamard's formula have order strictly greater than
the order of $[dY, bY]$. From this, we promptly conclude that $B$ cannot be a constant
multiple of $b$ contradicting Claim~2.

On the other hand, from the fact that $d (\partial b/\partial Y) - b (\partial d/\partial Y)$
vanishes identically, it follows that $d$ is a constant multiple of $b$. Therefore, we have
proved that every element in $D^{k-2} G$ has an infinitesimal generator of the form $c_i X + d_iY$
where $d_i$ is a constant multiple of the function $\overline{f}$ (the possibility
of having $d_i$ identically zero being clearly included in the discussion). The desired
contradiction is then obtained by considering the commutator of two elements $\varphi_1, \varphi_2
\in D^{k-2} G$. It follows from Lemma~\ref{formerclaim2} that the
infinitesimal generator of $\varphi_1 \circ \varphi_2 \circ \varphi_1^{-1}
\circ \varphi_2^{-1}$ is a multiple of the vector field $X$. Therefore $D^{k-1} G$ is abelian what contradicts the
assumption that $D^k G$ is not reduced to the identity. The proof of Proposition~\ref{commuting8} is over.
\end{proof}

Let then $G \subset \formdiffn$ be a {\it finitely generated}\, solvable non-abelian group. According to Proposition~\ref{commuting8},
every element in $D^1 G$ is contained in ${\rm Exp}\, (tX)$ for a certain $X \in \ghatX$. In other words, the infinitesimal generator of
every element in $D^1 G$ is a constant multiple of $X$. Consider a finite generating set $\{\psi_1, \ldots , \psi_k\}$ for $G$. The lemma below
complements Proposition~\ref{commuting8} by providing an explicit normal form for the generators $\psi_1, \ldots , \psi_k$.

\begin{lemma}
\label{normalformsolvablegroup}
Let $G \subset \formdiffn$ and $X$ be as above. Then there is a vector field $Y$ commuting with $X$ but not everywhere parallel to $X$ along with
a non-identically zero $\overline{f} \in \fieldC$ such that the following holds:
\begin{enumerate}
\item For every $i \in \{ 1, \ldots , k\}$, the infinitesimal generator $Z_i$ of $\psi_i$ has the form $a_i X + b_i Y$ where $a_i, b_i$ are first integrals
of $X$.

\item For every $i \in \{ 1, \ldots , k\}$, $b_i = \alpha_i \overline{f}$ with $\alpha_i \in \C$ (in particular $\overline{f}$ is itself a first integral of $X$).

\item Given $i,j \in \{ 1, \ldots , k\}$, then
$$
[Z_i, Z_j] = \overline{f}  \left( \frac{\partial (\alpha_i a_j -\alpha_j a_i)}{\partial Y} \right) X \, .
$$

\item Given $i,j \in \{ 1, \ldots , k\}$, then $\psi_i, \, \psi_j$ commute if and only if $[Z_i, Z_j] =0$ what, in turn, is equivalent to saying that
$\alpha_i a_j -\alpha_j a_i$ is a constant.

\item For every $\psi \in G$, the infinitesimal generator $Z$ of $\psi$ has order at $(0,0)$ less than or equal to the order of $X$.

\end{enumerate}
\end{lemma}

\begin{proof}
According to Proposition~\ref{commuting8} all elements in $D^1 G$ are contained in the center of $G$. Since all these elements
have $X$ as infinitesimal generator (up to a multiplicative constant), item~(1) follows from Lemma~\ref{commuting2}. Similarly,
for every $i,j \in \{ 1, \ldots , k\}$ the element $\psi_i \circ \psi_j \circ \psi_i^{-1} \circ \psi_j^{-1}$ lies in $D^1 G$ and hence possesses
an infinitesimal generator parallel to $X$. In view of it, Lemma~\ref{formerclaim2} implies item~(2) above. In turn, item~(3) becomes an
immediate computation. As to item~(4), Lemma~\ref{commuting1} says that $\psi_i, \psi_j$ commute if and only if $[Z_i, Z_j ] =0$. However,
item~(3) shows that $[Z_i,Z_j]=0$ if and only if $\alpha_i a_j -\alpha_j a_i$ is a first integral for $Y$. Since $\alpha_i a_j -\alpha_j a_i$
is also a first integral for $X$, the fact that $X,Y$ are not everywhere parallel ensures that $\alpha_i a_j -\alpha_j a_i$ must be constant
in this case.

It only remains to check item~(5). Suppose for a contradiction that $\psi \in G$ has an infinitesimal generator $Z$ whose
order at $(0,0) \in \C^2$ is strictly greater than the order of~$X$. Modulo adding $\psi$ to the generating set of $G$, we can assume
without loss of generality that $\psi=\psi_1$ so that $Z$ becomes $Z_1$. To prove item~(5), it suffices to find $j \in  \{ 2, \ldots , k\}$ so that
$\psi_j$ does not commute with $\psi_1$. In fact, in this case, the infinitesimal generator of $\psi_1 \circ \psi_j \circ \psi_1^{-1} \circ \psi_j^{-1}$ has
order strictly greater than the order of $X_1$ and, on the other hand, this infinitesimal generator is a constant multiple of $X$ what yields
the desired contradiction. Now, to check the existence of $\psi_j$ as desired, note that $\psi_j$ commutes with $\psi_1$ if and only if
$\alpha_1 a_j -\alpha_j a_1 \in \C$ (item~(4)). Thus, if $\psi_j$ commutes with $\psi_1$ for every $j \in  \{ 1, \ldots , k\}$, the condition
that $\alpha_1 a_j -\alpha_j a_1$ is a constant for every $j \in \{ 2, \ldots ,k\}$ implies that, indeed,
for every $i,j \in \{ 1, \ldots , k\}$, the value of $\alpha_j a_i -\alpha_i a_j$ is a constant as well. In other words $G$ is an abelian group what contradicts
the assumption that $G$ is solvable non-abelian. The lemma is proved.
\end{proof}

We are finally able to prove Proposition~\ref{commuting9}.

\begin{proof}[Proof of Proposition~\ref{commuting9}]
Consider the corresponding sequence of sets $S(j)$ and let $G(j)$ (resp. $\overline{G} (j,j-1)$) be the subgroup generated by
$S(j)$ (resp. $S(j) \cup S(j-1)$). Let $k$ be the largest integer for which $S(k)$ is not trivial.
Then $G(k)$ is abelian while $\overline{G} (k,k-1)$
is solvable. In particular, we can consider the {\it smallest}\, integer $m$ for which
$\overline{G} (m,m-1)$ is solvable. Let us assume for
a contradiction that $m \geq 2$ so that $\overline{G} (m,m-1)$ is strictly contained in $G$. Let $F$ be an element in $S (m-2)$
and note that, by construction, $F$ satisfies
$F^{\pm 1} \circ G (m) \circ F^{\mp 1}
\subset \overline{G} (m,m-1)$. Since $\overline{G} (m, m-1)$, and hence $G (m)$, are both solvable, it follows from
Proposition~\ref{commuting8} that they have non-trivial centers. These centers will respectively be denoted by $Z (\overline{G} (m, m-1))$ and $Z (G (m))$.

Another general remark concerning the groups $G(m)$ and $\overline{G} (m,m-1)$ is as follows. Let $\varphi_0$
be an element (not necessarily unique) of $S (m-1)$ having the smallest order of contact with the identity among all
elements in $S (m-1)$. Then every element in $S (m)$, and hence every element in $G (m)$, has contact order with the
identity strictly larger than the contact order of $\varphi_0$. In other words, there is an element $\varphi_0 \in S (m-1)$ whose order
of contact with the identity is strictly smaller than the orders of contact with the identity of all elements in $G (m)$.

Let us begin the discussion with the case where $\overline{G} (m,m-1)$ is an abelian group

\vspace{0.1cm}

\noindent {\sc Case A}: Suppose that the group $\overline{G} (m,m-1)$ is abelian.

The group $G (m) \subseteq \overline{G} (m,m-1)$ is abelian as well. Suppose that $G (m)$ is contained
in the span of two non everywhere parallel commuting vector fields (without being contained in the exponential
of a single vector field). In this case,
Lemma~\ref{commuting3} ensures the same must hold for $\overline{G} (m,m-1)$. In particular $F$ acts on the linear span $E$ of these vector
fields. Besides the eigenvalues of this action are equal to~$1$ since $F$ is tangent to the identity. The proof of Lemma~\ref{lastversionLemma3}
then shows that $F$ is naturally embedded in $E$. Hence $\overline{G} (m-1,m-2)$ is abelian and the desired
contradiction results at once.

Suppose now that every element in $G (m)$ has an infinitesimal generator of the form $aX$, for
a certain formal vector field $X$ and such that $a$ is a first integral for $X$. Since $G (m) \subseteq \overline{G} (m,m-1)$ and the latter group
is abelian, there are two possibilities for $\overline{G} (m,m-1)$, namely:
\begin{enumerate}
  \item $\overline{G} (m,m-1)$ is contained in the linear span of two non
  everywhere parallel commuting vector fields $Y,Z$.
  \item All elements in $\overline{G} (m,m-1)$ have infinitesimal generators of the form $aX$, where $a$ still is
  a first integral for $X$.
\end{enumerate}
Consider first the situation described in item~(1).
Since there is $\varphi_0 \in \overline{G} (m,m-1)$ whose order of contact with the identity is strictly smaller than the
orders of elements in $G (m)$, the inclusion $G (m) \subset  \overline{G} (m,m-1)$ ensures that $G (m)$ must be
contained in ${\rm Exp}\, (tX)$ for
a certain vector field, still denoted by $X$, belonging to the span in question. By construction, the order of $X$ at $(0,0) \in \C^2$
is strictly greater than the order of the remaining vector fields in the span of $Y,Z$ (apart from constant multiples of $X$).
Thus $F \in S (m-2)$ must take $X$ on $X$. It then follows that $F$ belongs to the centralizer
of ${\rm Exp}\, (tX)$. In other words, the set $S (m-2)$ is contained in the centralizer of ${\rm Exp}\, (tX)$.
Without loss of generality, there is an element $\varphi$ in $S (m-1)$ whose infinitesimal generator is $Y$.
Since $G (m)$ is not trivial, the collection of commutators $[\varphi ,F]$, for every $F \in S (m-2)$ is
contained in ${\rm Exp}\, (tX)$. To obtain the desired contradiction, we proceed as follows: let $F \in
S (m-2)$ be fixed. We can assume that $[\varphi ,F]$ is not the identity otherwise $F$ is contained in the exponential
of the span of $Y,Z$ ($F$ already commutes with ${\rm Exp}\, (tX)$). Hence, according
to Lemma~\ref{lastversionLemma1}, the infinitesimal generator of $F$ has the form $aX + bY$ where
$a,b$ are first integrals for $X$. Since the infinitesimal generator of $[\varphi ,F]$ is $X$, it follows from Lemma~\ref{formerclaim2}
that $b$ is a constant. A contradiction
then arises from observing that this ``normal form'' for
elements in $S (m-2)$ implies that the group generated by $S (m) \cup S (m-1) \cup S (m-2)$ is solvable (cf. again
Lemma~\ref{formerclaim2}). Thus the group $\overline{G} (m-1,m-2)$ is solvable as well and this is clearly impossible.

To complete the proof of Proposition~\ref{commuting9} in the case where $\overline{G} (m,m-1)$ is abelian,
it remains to check the case in which all elements in $\overline{G} (m,m-1)$ have infinitesimal generators
of the form $aX$. Now we have:

\noindent  {\it Claim~1}. $G (m)$ must be contained in ${\rm Exp}\, (tX)$.

\noindent {\it Proof of Claim~1}. Let $F$ be a given element in $S (m-2)$.
Suppose there are $\varphi_1 \in G (m) \cap {\rm Exp}\, (tX)$
and $\varphi_2 \in G (m) \cap {\rm Exp}\, (tAX)$, where $A$ is a non-constant first integral of $X$.
Denote by $r$ (resp. $s$) the order of $X$ (resp. $AX$) at the origin. As already seen, the collection of
vector fields of order $r$ (resp. $s$) inducing a non-trivial element in the abelian group $\overline{G} (m,m-1)$
is finite. Hence, up to passing to a finite power $F^k$ of $F$, it follows that $F^k$ fixed both
$X$ and $AX$, i.e. $F^k$ belongs to the centralizers of both ${\rm Exp}\, (tX)$ and ${\rm Exp}\, (tAX)$.
As previously seen, this implies that $F$ itself belongs to the intersections of the centralizers of
${\rm Exp}\, (tX)$ and ${\rm Exp}\, (tAX)$. Thanks to Lemma~\ref{lastversionLemma2},
the intersections of these centralizers is an abelian group
whose infinitesimal generators have all the form $aX$ ($a$ first integral of $X$), it follows again that
the group generated by $S (m-1) \cup S (m-2)$ is abelian. The resulting contradiction proves the claim.\qed

To conclude the proof, we still know that every $F \in S (m-2)$ belongs to the centralizer of
${\rm Exp}\, (tX)$. Let $S (m-2) = \{ F_1 , \ldots , F_l \}$ and let $c_i X + d_i Y$ denote the infinitesimal
generator of $F_i$, $i=1, \ldots ,l$. To obtain the desired contradiction, it suffices to show the following:

\noindent  {\it Claim~2}. For every pair $i,j$ such that $d_i . d_j$ is not identically zero,
the quotient $d_i /d_j$ is a constant.

Indeed, as already pointed out, Claim~2 implies that $G (m-1, m-2)$ is a solvable group what is impossible.

\noindent {\it Proof of Claim~2}. As already observed, without loss of generality there is
$\varphi$ in $S (m-1)$ whose infinitesimal generator is $Y$. The commutator of $F_i$ an $\varphi$ belongs to $G (m)$ and hence
must admit $X$ as infinitesimal generator (cf. Claim~$1$). The statement of Claim~2 becomes then a direct consequence
of Lemma~\ref{formerclaim2}.\qed

\vspace{0.1cm}

\noindent {\sc Case B}: Suppose that the group $\overline{G} (m,m-1)$ is solvable but not abelian.

According to Proposition~\ref{commuting8} the center of $\overline{G} (m,m-1)$ is non-trivial and all its elements admit a certain
vector field $X$ as infinitesimal generator. In particular $\overline{G} (m,m-1)$ is contained in the centralizer of ${\rm Exp}\, (tX)$ and,
in addition, there is a vector field $Y$ not everywhere parallel to $X$ and commuting with $X$.
Also, by construction, the group $G (m)$ contains the center of $\overline{G} (m,m-1)$. However, we recall that every element in
$\overline{G} (m,m-1)$ has an infinitesimal generator whose order is at most the order of $X$. Since these orders are preserved
by the adjoint action of elements in $\diffCtwo$, the condition $F^{\pm 1} \circ G (m) \circ F^{\mp 1} \subset \overline{G} (m,m-1)$
implies that every element $F\in S(m-2)$ must take $X$ to a constant multiple of $X$ and hence to $X$ itself since $F$ is tangent to
the identity. Thus $F$ belongs to the centralizer of $X$ and hence its infinitesimal generator has the form $aX +bY$ where $a,b$ are
first integrals of $X$.

Let us now consider the group $G (m)$. It was seen that $G (m)$ contains $D^1 \overline{G} (m,m-1)$ and hence elements whose
infinitesimal generator is $X$. On the other hand, recall that $G (m)$ is generated by elements having the form $[\psi, F]$ where
$\psi \in \overline{G} (m,m-1)$ and $F \in \overline{G} (m,m-1) \cup S (m-2)$. In view of the fact that $F$ belongs to the centralizer
of $X$, we conclude that $G (m)$ is also contained in the centralizer of $X$ since so is $\overline{G} (m,m-1)$. There are two cases to
be considered depending on whether or not $G (m)$ contains elements whose infinitesimal generators are not everywhere parallel to
$X$.

Suppose first that $G (m)$ contains an element whose infinitesimal generator $Y$ is not everywhere parallel to $X$. Then $G (m)$ contains
a rank~$2$ abelian group. Modulo passing to a finite power of $F$ this group must be preserved by the adjoint action of
$F$ since $F$ preserves $X$ and it also preserves the order of $Y$: up to multiplicative constants and additive constant multiples of $X$,
in the solvable group $\overline{G} (m,m-1)$ there can exist only finitely
many infinitesimal generators with the same order  since the difference between two of them has order
bounded by the order of~$X$, cf. Proposition~\ref{commuting8}.
Since a power of $F$ preserves $Y$ and $F$ is tangent to the identity,
it follows that $F$ itself preserves $Y$, cf. Lemma~\ref{commuting2}. It then follows from the discussion in Case~A, item~(1), that $F$ is
contained in the linear span of $X,Y$. Therefore $\overline{G} (m,m-1) \cup S (m-2)$ generated a solvable group what yields a
contradiction in the present case.

Suppose now that every element in $G (m)$ has an infinitesimal generator everywhere parallel to $X$. Since the infinitesimal generator
of each element $F \in S (m-2)$ has the form $aX +bY$, where $a,b$ are
first integrals of $X$, the fact that all the commutators $[\psi, F]$, where
$\psi \in \overline{G} (m,m-1)$ and $F \in \overline{G} (m,m-1) \cup S (m-2)$, have infinitesimal generators parallel to $X$ implies that
all the coefficients ``$b$'' differ by a multiplicative constant, cf. Lemma~\ref{formerclaim2}. It then follows that $\overline{G} (m,m-1) \cup S (m-2)$
still generates a solvable group what is impossible. Proposition~\ref{commuting9} is proved.
\end{proof}

\section{Proof of Theorem~C}

Building on the material developed in the previous section, and especially on Proposition~\ref{commuting9},
the proof of Theorem~C will be completed in this last section. Let us first make use of Ghys's observation \cite{ghysBSBM}
concerning convergence of commutators for diffeomorphisms ``close to the identity'' to establish the following proposition:

\begin{prop}
\label{almostthere}
Suppose that $G \subset \diffCtwo$ is a group possessing locally discrete orbits. Then $G$ is solvable.
\end{prop}

\begin{proof}
Consider a finite set $S$ consisting of tangent to the identity local diffeomorphisms of $(\C^2, 0)$. Suppose that the group $G$ generated
by the set $S$ is not solvable (at level of groups of germs of diffeomorphisms). Then consider the pseudogroup generated by $S$ on
a certain (sufficiently small) neighborhood of the identity which will be left implicit in the subsequent discussion for the sake of notation.
The proof of the proposition amounts to showing that the resulting pseudogroup
$G$ is {\it non-discrete}\, in the sense that it contains a sequence of elements $h_i$ satisfying the following conditions:
\begin{itemize}
\item $h_i \neq {\rm id}$ for every $i \in \N$ and, furthermore, as element of the pseudogroup $G$, $h_i$ is defined on a ball $B_{\epsilon}$ of uniform radius $\epsilon > 0$
about $(0,0) \in \C^2$.

\item The sequence of mappings $\{ h_i \}$ converges uniformly to the identity on $B_{\epsilon}$.
\end{itemize}
Assuming the existence of a sequence $h_i$ as indicated above, it follows that each of the sets ${\rm Fix}_i = \{ p \in B_{\epsilon} \, \; ; \; \, h_i (p) = p \}$
is a proper analytic subset of $B_{\epsilon}$. For every $N \geq 1$, pose $A_N = \bigcap_{i=N}^{\infty} {\rm Fix}_i$ so that $A_N$ is also a proper analytic set
of $B_{\epsilon}$. Finally, let $F = \bigcup_{N=1}^{\infty} A_N$. The set $F$ has null Lebesgue measure so that points in $B_{\epsilon} \setminus F$ can be
considered. If $p \in B_{\epsilon} \setminus F$ then, by construction, there is a subsequence of indices $\{ i(j) \}_{j \in \N}$ such that $h_{i(j)} (p) \neq p$ for every
$j$. Since $h_i$ converges to the identity on $B_{\epsilon}$, the sequence $\{ h_{i(j)} (p) \}_{j \in \N}$ converges non-trivially to~$p$. This show that the
orbit of $p$ is not locally discrete and establishes the proposition modulo verifying the existence of mentioned sequence $\{ h_i \}$.

The construction of the sequence $\{ h_i \}$ begins with an estimate concerning commutators of diffeomorphisms that can be found
in \cite{lorayandI}, page~159, which is itself similar to another estimate found in \cite{ghysBSBM}. Let $F_1, F_2$ be local diffeomorphisms (fixing the origin
and) defined on the ball $B_r$ of radius $r > 0$ about the origin of $\C^2$. For small $\delta > 0$,
to be fixed later, suppose that
\begin{equation}
\max \{ \sup_{z \in B_r} \Vert F_1^{\pm 1} (z) -z \Vert \; , \; \sup_{z \in B_r} \Vert F_2^{\pm 1} (z) -z \Vert \} \leq \delta/4 \, . \label{initializing}
\end{equation}
Then, given $0 < \tau \leq 2 \delta$, the commutator $[F_1,F_2]$ is defined on the ball of radius $r -4\delta -\tau$ and, in addition, it verifies the estimate
\begin{equation}
\sup_{z \in B_{r -4\delta -\tau}} \Vert [F_1,F_2] (z) -z \Vert \leq \frac{2}{\tau}  \sup_{z \in B_r} \Vert F_1 (z) -z \Vert \, . \,
 \sup_{z \in B_r} \Vert F_2 (z) -z \Vert \, . \label{estimateLorayandI}
\end{equation}

Let us apply the preceding estimate to elements in $S(i)$. Because $G$ consists of diffeomorphisms tangent to the identity, modulo conjugating
it by a homothety of type $(x,y) \mapsto (\lambda x, \lambda y)$, with $\vert \lambda \vert < 1$, all local diffeomorphisms in $S$ can be supposed
to be defined on the unit ball. Furthermore they can also be supposed to satisfy
Estimate~(\ref{initializing}) for $r=1$ and some arbitrarily small $\delta >0$ to be fixed later. Setting $\tau = 2\delta$, it then follows that
every element $H$ in $S(1)$ is defined on $B_{1- 6 \delta}$ and satisfies
$$
\sup_{z \in B_{1 - 6\delta}} \Vert H (z) -z \Vert \leq \delta /2^4\, .
$$
Next, note that every element in $S(2)$ is the commutator of an element in $S (1)$ and an element in $S \cup S(1)$. Thus, applying again
Estimate~(\ref{estimateLorayandI}) to $r=1-6\delta$, $\delta_1 =\delta/2$ and $\tau_1 = \tau/2 = \delta$, we conclude that every element
$H$ in $S(2)$ is defined on $B_{r_1}$, where $r_1 = 1 -6\delta (1 + 1/2)$. Furthermore these elements $H$ satisfy the estimate
$$
\sup_{z \in B_{r_1}} \Vert H (z) -z \Vert \leq \delta /2^5 \, .
$$
Continuing inductively with $r_i = 1 -6\delta (\sum_{n=0}^i 1/2^n)$, $\delta_i = \delta_{i-1}/2$ and $\tau_i = \tau_{i-1}/2 = \delta_{i-1}$, we conclude that
every element $H^i$ in $S (i)$ is defined on a ball of radius $1 - 12 \delta$ and satisfy $\sup_{z \in B_{1-12\delta}} \Vert H^i (z) - z\Vert \leq \delta / 2^{i+3}$
In particular, if $\delta = 1/24$, all elements in $S(i)$ are defined on the ball of radius~$1/2$
($i \in \N$). Similarly, it is also clear that elements in $S (i)$ converge uniformly to the identity on $B_{1/2}$. Therefore, to obtain the desired sequence $h_i$,
it suffices to pick, for every $i$, one element $h_i \in S(i)$ which is different from the identity. In view of
Proposition~\ref{commuting9}, the sequence
of sets $S (i)$ never degenerate into the identity alone so that the indicated choice of $h_i$ is always possible. The proof of the proposition is over.
\end{proof}

We are finally ready to prove Theorem~C.

\begin{proof}[Proof of Theorem~C]
Let $\Diffgentwo$ denote the group of (germs of) holomorphic diffeomorphisms at $(0,0) \in \C^2$.
Consider a subgroup $G \subset \Diffgentwo$ possessing locally discrete orbits in some neighborhood $U$
of $(0,0) \in \C^2$. Let $\rho$ be the homomorphism from $G$
to ${\rm GL}\, (2,\C)$ assigning to an element $\varphi \in G$ its Jacobian matrix at the origin. Denoting by $\Gamma \subset
{\rm GL}\, (2,\C)$ the image of $\rho$, let us consider the short exact sequence
$$
0 \longrightarrow G_0 = {\rm Ker}\, (\rho) \longrightarrow G \longrightarrow \Gamma \longrightarrow 0 \, .
$$
The kernel $G_0$ of $\rho$ consists of those elements in $G$ that are tangent to the identity. Since $G$, and hence $G_0$, has
locally discrete orbits, it follows from Proposition~\ref{almostthere} that $G_0$ is solvable. Therefore, to conclude that $G$ is solvable,
it suffices to check that the assumption of having locally discrete orbits forces $\Gamma$ to be solvable as well.

While $\Gamma$ is a subgroup of ${\rm GL}\, (2,\C)$, its standard action on $(\C^2, 0)$ has little to do with the action of $G$. In fact,
if $\gamma$ is an element of $\Gamma$, then $\gamma$ is simply the derivative at the origin of an actual element $\varphi \in G$ and it is $\varphi$, rather than $\gamma$,
that acts on $(\C^2, 0)$. Thus, the effect of the non-linear terms in $\varphi$ must be taken into account.

Recall that ${\rm PSL}\, (2, \C)$ is the quotient of the subgroup ${\rm SL}\, (2,\C)$ of ${\rm GL}\, (2,\C)$ consisting of matrices whose determinant equals~$1$
by its center which, in turn, consists of $\{ I , -I \}$ where $I$ stands for the identity matrix. Let us consider the projection of
$\Gamma$ in ${\rm PSL}\, (2, \C)$ and let ${\rm P} G$ denote its image.

\vspace{0.1cm}

\noindent {\it Claim 1}. Without loss of generality, we can suppose that ${\rm P} G$ is not solvable.

\noindent {\it Proof of Claim 1}. Note that ${\rm P} G$ is solvable if and only if its first derived group $D^1 ({\rm P} G)$ is abelian.
Now, denote by $\widetilde{{\rm P} G}$ the projection of $\Gamma$ to ${\rm SL}\, (2,\C)$ as an intermediate step for the projection
of $\Gamma$ onto ${\rm P} G$. The first derived group of $\widetilde{{\rm P} G}$ will be denoted by $D^1( \widetilde{{\rm P} G})$.
Naturally the group $D^1 (\widetilde{{\rm P} G})$ must be abelian provided that
$D^1 ({\rm P} G)$ is abelian. In fact, if two matrices $A, B$ commute, then the same applies to any combination of $\pm A, \, \pm B$.
On the other hand, $D^1 (\widetilde{{\rm P} G})$ coincides with $D^1 \Gamma$ since the determinant of the commutator of two matrices
necessarily equals~$1$. Hence the group $\Gamma$ itself is abelian and the theorem is proved in this case.\qed

\vspace{0.1cm}

Next note that, as a subgroup of ${\rm PSL}\, (2, \C)$, ${\rm P} G$ may or may not be discrete. Suppose ${\rm P} G$ non-discrete.
Being, in addition, non-solvable, it follows that ${\rm P} G$ is dense in ${\rm PSL}\, (2, \C)$. In particular, it contains non-elementary
discrete Kleinian groups (or even Schottky groups). So it is sufficient to show that a group $G \subset \Diffgentwo$ cannot have
locally discrete orbits provided that derivatives at $(0,0) \in \C^2$
of its elements induce a non-elementary Kleinian group in ${\rm PSL}\, (2, \C)$. This will be done below.

Summarizing what precedes, the group ${\rm P} G$ can be supposed to be a non-elementary discrete subgroup of ${\rm PSL}\, (2, \C)$, i.e. ${\rm P} G$ is
a non-elementary Kleinian group. Under this assumption, we need to prove that the corresponding group $G \subset \Diffgentwo$ does not
have locally discrete orbits.
The condition of having a non-elementary Kleinian group ${\rm P} G$ will be exploited through the fact that these
groups always possess loxodromic elements, see \cite{apanasov}. Let us first consider the meaning of loxodromic elements in our context.

Consider an element $\varphi \in G$ whose derivative $D_0 \varphi$ at the origin gives rise to a loxodromic element in ${\rm P} G$. Then
$D_0 \varphi$ is diagonalizable. Note also that the Jacobian determinant of $D_0 \varphi$ can be supposed equal to~$1$ since, again,
we can start out by looking at $D^1 \Gamma$, instead of $\Gamma$, and the former group still induces a non-elementary Kleinian
group in ${\rm PSL}\, (2, \C)$. Therefore, the eigenvalues of $D_0 \varphi$ are $\lambda$ and $\lambda^{-1}$, with $\vert \lambda \vert > 1$.
It follows that $\varphi$ has a hyperbolic fixed point at the origin with stable and unstable manifolds, $W^s_{\varphi}, \, W^u_{\varphi}$,
having complex dimension~$1$ and intersecting transversely at $(0,0) \in \C^2$. Fix then a {\it closed annulus}\, $A^s \subset W^s_{\varphi}$
(resp. $A^u \subset W^u_{\varphi}$) with radii $r_2 > r_1 > 0$ such that every point $p \in W^s_{\varphi} $ (resp. $p \in W^s_{\varphi} $)
possesses an orbit by $\varphi$ non-trivially intersecting $A^s$ (resp. $A^u$).

Given a point $p$ in a fixed neighborhood $U$ of the origin where the group $G$ has locally discrete orbits, denote by
$\calO_G (p)$ the orbit of $p$ (by the pseudogroup) $G$. Similarly, let ${\rm Acc}_p (G)$ denote the set of {\it ends}\,
of $\calO_G (p)$. In other words, if $\calO_G (p)$ is infinite and $p =p_1, p_2, \ldots$ is an enumeration of its points,
then ${\rm Acc}_p (G) = \bigcap_{n=1}^{\infty} [ \overline{\calO_G (p)} \setminus \bigcup_{j=1}^n \{ p_j\} ]$. If $\calO_G (p)$ is finite,
then ${\rm Acc}_p (G) = \emptyset$. Clearly ${\rm Acc}_p (G)$ is closed and invariant by $G$ (viewed as pseudogroup). The following claim is
the key for the proof of Theorem~B.

\vspace{0.1cm}

\noindent {\it Claim 2}. For every point $p \in A^s$, the closed set $A^s \cap {\rm Acc}_p (G)$ is not empty.

\vspace{0.1cm}

Note that Claim~2 does not immediately imply Theorem~C for it does not assert that $p$ itself belongs to $A^s \cap {\rm Acc}_p (G)$.
However, if this were the case, then clearly the orbit of $p$ would not be locally discrete. The resulting contradiction would then ensure that ${\rm P}\,G$
cannot contain a non-elementary Kleinian group so that the statement of Theorem~C would follow. However, by resorting to a standard
application of Zorn Lemma, Claim~2 can still be used to prove Theorem~C. Let us first provide the details and then go back to the proof
of Claim~2. To begin with, if $K \subseteq A^s$ is a non-empty closed set, we shall say that $K$ is {\it relatively invariant}\, by the pseudogroup $G$
if, for every point $p \in K$ and every point $q \in A^s \cap {\rm Acc}_p (G)$, the point $q$ lies in $K$ as well. Next, let
$\mathfrak{C}$ denote the collection of non-empty closed sets in $A^s$ that are relatively invariant by the pseudogroup $G$. Claim~2 ensures that the collection
$\mathfrak{C}$ is not empty. In fact, $A^s \cap {\rm Acc}_p (G)$ in a non-empty set relatively invariant under $G$, and thus $A^s \cap {\rm Acc}_p (G)$
belongs to $\mathfrak{C}$ for every $p \in A^s$. Now, let the collection
$\mathfrak{C}$ be endowed with the partial order defined by inclusion. Finally, given a sequence $K_1 \supset K_2 \supset \ldots$ of
sets in $\mathfrak{C}$, the intersection $K_{\infty}=\bigcap_{i=1}^{\infty} K_i$ is non-empty since each $K_i$ is compact (closed and contained in the compact
set $A^s$). The set $K_{\infty}$ belongs to $\mathfrak{C}$, since it is clearly relatively invariant by $G$, and satisfies $K_{\infty} \subset K_i$ for every~$i$. According
to Zorn Lemma, the collection $\mathfrak{C}$ contain minimal elements, so that we can consider a minimal element $K$. Choose then $q \in K$ and consider
the non-empty set $A^s \cap {\rm Acc}_q (G)$. If $q \not\in {\rm Acc}_q (G)$, then $A^s \cap {\rm Acc}_q (G)$ would be an element of $\mathfrak{C}$
strictly smaller than $K$. The resulting contradiction shows that $q \in A^s \cap {\rm Acc}_q (G)$ and finishes the proof of Theorem~C.\qed

It only remains to prove Claim~2.

\vspace{0.1cm}

\noindent {\it Proof of Claim~2}. Recall that $A^s \subset W^s_{\varphi}$
(resp. $A^u \subset W^u_{\varphi}$) is an annulus such that every $p \in W^s_{\varphi} $ (resp. $p \in W^s_{\varphi} $)
possesses an orbit by $\varphi$ non-trivially intersecting $A^s$ (resp. $A^u$).

Now consider another element $\psi \in G$ whose Jacobian matrix at the origin is hyperbolic with determinant equal to~$1$. Again stable and
unstable manifolds for $\psi$ will respectively be denoted by $W^s_{\psi}, \, W^u_{\psi}$. Since a Kleinian group contains ``many'' loxodromic
elements, $\psi$ can be chosen so that the four manifolds $W^s_{\varphi}, \, W^u_{\varphi}, \, W^s_{\psi}, \, W^u_{\psi}$ intersect pairwise transversely
at the origin. The previously fixed annuli $A^s \subset W^s_{\varphi}$ and $A^u \subset W^u_{\varphi}$ will be denoted in the sequel by
$A^s_{\varphi}$ and $A^u_{\varphi}$. An annulus $A^s_{\psi} \subset W^s_{\psi}$ (resp. $A^s_{\psi} \subset W^s_{\psi}$)
with analogous properties concerning $\psi$ is also fixed. To prove the claim it suffices to check that every point $p$ in $A^s_{\varphi}$ is such that
$A^u_{\psi} \cap {\rm Acc}_p (G) \neq \emptyset$. Indeed, let $p^{\ast} \in A^u_{\psi}$ be a point in $A^u_{\psi} \cap {\rm Acc}_p (G)$. The analogue
argument changing the roles of $\varphi, \, \psi$ and replacing them by their inverses, will ensure that $A^s_{\varphi} \cap {\rm Acc}_{p^{\ast}} (G)
\neq \emptyset$. Since $p^{\ast}$ lies in ${\rm Acc}_p (G)$ and this set is invariant under the pseudogroup $G$, it will follow that
$A^s \cap {\rm Acc}_p (G) \neq \emptyset$ as desired.

Finally to check that $A^u_{\psi} \cap {\rm Acc}_p (G) \neq \emptyset$ for every point $p \in A^s_{\varphi}$, we proceed as follows.
Consider local coordinates $(x,y)$ about the origin of $\C^2$ so that $\{ x=0\} \subset W^u_{\psi}$ and $\{ y=0\} \subset W^s_{\psi}$. Recall that
$W^s_{\varphi}$ is smooth and intersects the coordinate axes transversely at the origin. Since this intersection is transverse, we can assume
that it is the only intersection point of $W^s_{\varphi}$ with the coordinate axes. In particular, a point $p \in A^s_{\varphi}$ has coordinates
$(u,v)$ with $u.v \neq 0$. By iterating $\varphi$, we can find points $p_n = (u_n , v_n) =\varphi^n (p) \in \C^2$ such that $\vert u_n \vert \rightarrow 0$
and
$$
\frac{1}{C} \vert u_n \vert \leq \vert v_n \vert \leq C \vert u_n \vert \, ,
$$
for some uniform constant $C$ related to the ``angles'' between $W^s_{\varphi}$ and the coordinate axes at the origin. Now, for every $n$,
consider the points of the form $\psi (p_n) , \ldots , \psi^{l(n)} (p_n)$ where $l(n)$ is the smallest positive integer for which the absolute value
of the second component of $\psi^{l(n)} (p_n)$ is greater than $\sup_{z \in A^u_{\psi}} \vert z \vert$. The integer $l(n)$ exists since $\psi$
has a hyperbolic fixed point at the origin and the action of $\psi$ on $p_n$ is such that the first coordinate becomes smaller and smaller while
the second coordinate gets larger and larger.
Now it is clear that the set $\bigcup_{n=1}^{\infty} \{ \psi (p_n) , \ldots , \psi^{l(n)} (p_n) \}$ accumulates on $A^u_{\psi}$ and this ends the proof
of Claim~2.
\end{proof}


\bigskip

\bigskip

\begin{flushleft}
{\sc Julio Rebelo} \\
Institut de Math\'ematiques de Toulouse\\
118 Route de Narbonne\\
F-31062 Toulouse, FRANCE.\\
rebelo@math.univ-toulouse.fr

\end{flushleft}

\bigskip

\begin{flushleft}
{\sc Helena Reis} \\
Centro de Matem\'atica da Universidade do Porto, \\
Faculdade de Economia da Universidade do Porto, \\
Portugal\\
hreis@fep.up.pt \\

\end{flushleft}

\end{document}